\documentclass[final,leqno]{siamltex}
\usepackage[utf8]{inputenc}
\usepackage[bookmarksopen,bookmarksopenlevel=0,bookmarksdepth=2]{hyperref}
\usepackage{amsmath}
\usepackage{amssymb}
\usepackage{graphicx}
\usepackage{doi}
\usepackage{cleveref}
\usepackage[usenames]{color}
\usepackage{wrapfig}
\usepackage{stmaryrd}
\usepackage{array} 
\usepackage{mathrsfs}
\usepackage{multirow}
 \usepackage{txfonts}

\SetSymbolFont{stmry}{bold}{U}{stmry}{m}{n}

\newtheorem{lem}{Lemma}[section]
\newtheorem{rem}{Remark}[section]

\usepackage{enumitem}
\setenumerate[1]{itemsep=0pt,partopsep=0pt,parsep=\parskip,topsep=5pt}
\setitemize[1]{itemsep=0pt,partopsep=0pt,parsep=\parskip,topsep=5pt}
\setdescription{itemsep=0pt,partopsep=0pt,parsep=\parskip,topsep=5pt}

\title{
Structure-preserving weighted BDF2 methods for Anisotropic Cahn-Hilliard model: uniform/variable-time-steps}

\author{Meng Li\thanks{School of Mathematics and Statistics, Zhengzhou University, Zhengzhou 450001, China.
Email: limeng@zzu.edu.cn. }
\and Jingjiang Bi\thanks{School of Mathematics and Statistics, Zhengzhou University, Zhengzhou 450001, China. }
\and Nan Wang\thanks{School of Mathematics and Statistics, Zhengzhou University, Zhengzhou 450001, China.
Corresponding author: nwang@zzu.edu.cn. }
}

\date{}
\begin{document}
\sloppy
\maketitle

\begin{abstract}
In this paper, we innovatively develop uniform/variable-time-step weighted and shifted BDF2 (WSBDF2) methods for the anisotropic Cahn-Hilliard (CH) model, combining the scalar auxiliary variable (SAV) approach with two types of stabilized techniques.
Using the concept of $G$-stability,
the uniform-time-step WSBDF2 method is theoretically proved to be energy-stable. Due to the inapplicability of the relevant G-stability properties, another technique is adopted in this work to demonstrate the energy stability of the variable-time-step WSBDF2 method. In addition, the two numerical schemes are all mass-conservative.
 Finally, numerous numerical simulations are presented to demonstrate the stability and accuracy of these schemes.
\end{abstract}





\section{Introduction}
As the typical phase-field, the Cahn-Hilliard (CH) equation, initially introduced to describe phase separation in binary alloys, is pivotal in materials science, especially for elucidating the qualitative features of two-phase systems under conditions of isotropy and constant temperature.
Recently, the CH model has found widespread applications beyond its initial purpose, extending to surface diffusion of adatoms in stressed epitaxial thin films \cite{r2006surface,eggleston2002ordered,wise2005quantum}, dealloying in corrosion processes \cite{erlebacher2001evolution}, vesicle dynamics \cite{campelo2006dynamic}, tumor growth \cite{aristotelous2015adaptive,frieboes2007computer}, multiphase fluid flow \cite{lowengrub1998quasi,kim2004conservative,badalassi2003computation}, bacterial films formation \cite{klapper2006role} and image processing \cite{bertozzi2006inpainting,dolcetta2002area,chalupecky2004numerical}. More applications and numerical methods about the phase-field model can refer to Refs.  \cite{miranville2019cahn,novick2008cahn,du2019maximum,du2021maximum}.

The anisotropic CH equation, proposed by Torabi et al. \cite{wise2007solving,torabi2009new}, aims to describe the phenomenon of faceted pyramid formation on nanoscale crystal surfaces. In the anisotropic system, there exists a term $\gamma(\nabla\phi/|\nabla\phi|)$ that changes its sign as $|\nabla\phi|$ is close to zero, leading to severe oscillations and posing significant challenges to the development of algorithms for this system.
In \cite{wise2007solving,torabi2009new}, fully-implicit methods were employed to handle the nonlinear terms in the anisotropic CH system.
However, the methods lack the proof of energy stability, and are computationally expensive due to the iterative requirement of the implicit schemes.
In \cite{chen2013efficient,shen2018stabilized}, using stablilization technique, the Shen et al. established energy stable schemes for the isotropic and anisotropic CH models. All the nonlinear terms are explicitly handled by these schemes, which significantly enhances computational efficiency.
However, the proof of energy stability only limits to the isotropic case.
Yang et al. \cite{chen2019fast,xu2019efficient} developed
second-order stabilized scalar auxiliary variable (SAV) method and stabilized invariant energy 1uadratization (IEQ) method for the anisotropic CH equation, and demonstrated the energy stability of the schemes. As a development of \cite{chen2019fast},
the first contribution of this work is the construction of uniform-time-step weighted and shifted BDF2 method integrated with the SAV approach for the anisotropic CH model, and derive the mass conservation and energy stability of the numerical scheme by applying the property of G-stability \cite{hairer1996g}. The work \cite{chen2019fast} is a special case of this paper, with the weighted parameter $\theta=1$. Nonetheless, there exist substantial differences in the theoretical analytical approaches.

In recent years, variable-time-step methods are popular to solve nonlinear partial differential equations (PDEs) in physical problems.Compared with the uniform temporal mesh for PDEs, variable-time-step methods and adaptive time-stepping techniques \cite{gomez2011provably,liao2021analysis1,qiao2011adaptive,zhang2012adaptive,di2022sharp,zhang2022sharp} could better capture the multi-scale behaviors of the solutions in long time simulation and  improve the computational efficiency under the same accuracy of numerical schemes. For variable-step BDF2 method, Grigorieff \cite{grigorieff1983stability} proved that the zero-stability of this method for ODEs under the adjacent time step ratio $\gamma_{n+1}:=\tau_{n+1}/\tau_n<1+\sqrt{2}$ . In \cite{becker1998second}, Becker conducted a rigorous stability and convergence analysis of the variable-step BDF2 scheme for linear parabolic PDEs under the restriction $\gamma_{n}<(2+\sqrt{13})/3\approx 1.8685$. Chen at el. \cite{chen2019second} studied the numerical scheme for CH model under the time step ratio $\gamma_n<3.561$. Liao and Zhang \cite{liao2021analysis} by using discrete orthogonal convolution (DOC) kernel technology, they studied the condition with $0<r_{k}\leq 4.864$.
  Zhang and Zhao \cite{zhang2021sharp}  obtained a new adjacent time-step ratio
  $0< r_{k} \leq r_{\max}-\delta$ for any small constant $0<\delta<r_{\max}\approx4.8645$ by using the properties of DOC kernel  and  discrete complementary convolution (DCC) kernel. Recently, Hou and Qiao \cite{hou2023implicit} proposed a variable-time-step BDF2 SAV scheme
to solve the phase field crystal equation under the maximum time-step ratio $0<r_{\max}\leq 4.864$.
Interested readers can refer to \cite{zhang2022sharp,liao2021analysis,liao2018sharp,MR4206656,chen2021weighted,akrivis2024variable,liu2023positivity,li2023variable,wang2023unconditional,zhang2013adaptive} for the variable-time-step methods applied to the divergent models, including parabolic equation, phase-field model and nonlinear Ginzburg-Landau equation. Up to now, there exists no work focusing on the variable-time-step methods for the anisotropic CH system, which is another contribution of this work.

To sum up, the main contributions of this work include: \\
(i) This work applies the uniform-time-step WSBDF2 method to solve the anisotropic CH model. We extend the BDF2 method in \cite{chen2019fast} to a more general form, allowing us to obtain various discretization schemes by adjusting the value of $\theta$. Additionally, by incorporating the concept of G-stability \cite{hairer1996g}, we demonstrate the energy stability of the proposed uniform-time-step WSBDF2 method.\\
(ii)
The other main contribution of this work is to design a new structure-preserving method
by combining the variable-time-step WSBDF2 method with the SAV approach proposed in \cite{huang2020highly}. The adopted SAV approach differs from the traditional one \cite{shen2018scalar}, in order to construct the energy-stable algorithm based on variable-time-step method. By using different analytical technique as the uniform case, we also successfully obtain the mass conservation and energy stability of the constructed variable-time-step WSBDF2 method.
\\
(iii) Furthermore, in order to eliminate the ill-posedness of the anisotropic model,  it is indeed necessary to incorporate regularization terms in the continuous systems, especially for the strongly anisotropic case. The regularized models include the linear regularization and the Willmore regularization.
In the process of practical computation, aiming to reduce the oscillations caused by the anisotropic surface energy function $\gamma(\cdot)$,
we develop two stabilized methods for the uniform/variable-time-step WSBDF2 methods by adding two types of stabilization terms.
The numerical experiments demonstrate that adding stabilization terms can maintain numerical stability without affecting the accuracy and structure-preservation of the solutions.

The rest of the paper is organized as follows. In Section \ref{section_2}, We provide a brief introduction to the anisotropic CH model and its regularized systems. In Section \ref{section_3}, two uniform-time-step numerical schemes are developed to the regularized systems,
and their mass-conservation and energy-stability properties are rigorously proven.
In Section \ref{section_4}, we develop two variable-time-step methods for solving the regularized systems, and also demonstrate their structure-preserving properties. Several numerical examples are given in section \ref{section_5}.
Some conclusions follow in Section \ref{section_6}.

\section{Anisotropic Cahn–Hilliard model equation and its energy law}\label{section_2}
In this section, we provide a brief introduction to the anisotropic CH model. Let $\Omega\in\mathbb{R}^{d}\;(d=2,3)$ be a smooth, open, bounded and connected domain, and let $\phi$ be an order parameter that takes values of $\pm1$ in the two phases with a smooth transition layer of thickness $\epsilon$. The free energy density of Ginzburg-Landau-type is given by
 \begin{align}
    \mathcal{F}(\phi)=\gamma(\mathbf{n})\left(\frac{1}{2}|\nabla\phi|^2+\frac{1}{\epsilon^{2}}F(\phi)\right),
 \end{align}
 where $\gamma(\mathbf{n})$ is a function describing the anisotropic property and $\mathbf{n}$ is the interface normal
defined by
\begin{align}
    \mathbf{n}=\left(n_1,n_2\right)^T=\frac{\nabla\phi}{|\nabla\phi|}\qquad or \qquad\mathbf{n}=\left(n_1,n_2,n_3\right)^T=\frac{\nabla\phi}{|\nabla\phi|},
\end{align}
and the energy density function takes the usual double-well form
\begin{align}
    F(\phi)=\frac{1}{4}\left(\phi^2-1\right)^2.
\end{align}
Then, the surface free energy of the system is as follow
\begin{align}\label{eqn:free_energy_1}
    \mathcal{E}(\phi)=\int_{\Omega}\mathcal{F}d\Omega.
\end{align}
The difference between isotropic system and  anisotropic system lies in the choice of $\gamma(\mathbf{n})$. When $\gamma(\mathbf{n})\equiv1$, the system reflects isotropic properties. In the case of anisotropy, $\gamma(\mathbf{n})$ varies with $\mathbf{n}$ in a nontrivial way. In this paper, we consider the fourfold symmetric anisotropic function
\begin{align}\label{eqn:anisotropic_property}
    \gamma(\mathbf{n})=1+\alpha \cos(4\vartheta)=1+\alpha\left(4\sum^{d}_{i=1}n_i^4-3\right),
\end{align}
where $\vartheta$ represents the orientation angle of the interfacial normal to the interface, and the non-negative parameter $\alpha$ in \eqref{eqn:anisotropic_property} describes the intensity of anisotropy. When $\alpha=0$, an isotropic system is attained, namely, the free energy shows no preference for any orientations. When $\alpha\ge\frac{1}{15}$, the system becomes strongly anisotropic; i.e., the underlying Phase-filed equation is ill-posed. In order to regularize the original problem, an additional term $\mathrm{G}(\phi)$, is introduced into \eqref{eqn:free_energy_1} to penalize infinite curvatures in the resulting corners. As a result, the total system energy is modified as follow
\begin{align}\label{eqn:total_energy}
    \mathcal{E}(\phi)=\int_{\Omega}\left(\mathcal{F}+\frac{\beta}{2}\mathrm{G}\right)d\Omega,
\end{align}
where $\beta$ is a regularization parameter. We consider two types of regularization with two distinct forms of $\mathrm{G}$.\par
The first type is the linear regularization based on the bi-Laplacian of the phase variable
\begin{align}
    \mathrm{G}(\phi)=(\Delta\phi)^2,
\end{align}
and the second one is the nonlinear Willmore regularization with
\begin{align}
    \mathrm{G}(\phi)=\left(\Delta\phi-\frac{1}{\epsilon^2}f(\phi)\right)^2,
\end{align}
where $f(\phi)=F'(\phi)=\phi(\phi^2-1)$. By taking the $H^{-1}$ gradient flow on the total system free energy \eqref{eqn:total_energy}, the anisotropic CH system with linear regularization is given by
\begin{subequations}\label{eqn:linear_gradient_flow}
\begin{align}
    \phi_t&=\nabla\cdot(M(\phi)\nabla\mu),\label{eqn:linear_1}\\
    \mu&=\frac{1}{\epsilon^2}\gamma(\mathbf{n})f(\phi)-\nabla\cdot\mathbf{m}+\beta\Delta^2\phi;\label{eqn:linear_2}
\end{align}
\end{subequations}
and the system with Willmore regularization is represented by
\begin{subequations}\label{eqn:willmore_gradient_flow}
    \begin{align}
       \phi_t&=\nabla\cdot(M(\phi)\nabla\mu),\label{eqn:willmore_1}\\
       \mu&=\frac{1}{\epsilon^2}\gamma(\mathbf{n})f(\phi)-\nabla\cdot\mathbf{m}+\beta\left(\Delta-\frac{1}{\epsilon^2}f'(\phi)\right)\left(\Delta\phi-\frac{1}{\epsilon^2}f(\phi)\right),\label{eqn:willmore_2}
    \end{align}
\end{subequations}
where chemical potential $\mu$ is the variational derivative of $\mathcal{E}(\phi)$, $f'(\phi)=3\phi^2-1$, $M(\phi)\geq M_0>0$ is the mobility function that depends on the phase variable $\phi$ or a constant, and the vector
field $\mathbf{m}$ is given by
\begin{align}
    \mathbf{m}=\gamma(\mathbf{n})\nabla\phi+\frac{\mathbb{P}\nabla_\mathbf{n}\gamma(\mathbf{n})}{|\nabla\phi|}\left(\frac{1}{2}|\nabla\phi|^2+\frac{1}{\epsilon^2}F(\phi)\right)
\end{align}
with $\mathbb{P}=\mathbb{I}-\mathbf{n}\mathbf{n}^T$. Without loss of generality, periodic boundary conditions are employed to eliminate the complexities associated with the boundary integrals. We also note that the boundary conditions can be of the no-flux type as
\begin{align}
    \left.\frac{\partial\phi}{\partial\mathbf{n}}\right|_{\partial\Omega}=\left.\frac{\partial\mu}{\partial\mathbf{n}}\right|_{\partial\Omega}=\left.\frac{\partial\omega}{\partial\mathbf{n}}\right|_{\partial\Omega}=0,
\end{align}
where $\omega=\Delta\phi$ for the system with linear regularization and $\omega=\Delta\phi-\frac{1}{\epsilon^2}f(\phi)$ for the system with Willmore regularization .\par
For the linear regularization system, by taking the $L^2$ inner product of equation \eqref{eqn:linear_1} with $-\mu$ and equation \eqref{eqn:linear_2} with $\phi_t$, employing integration by parts, and subsequently combining the two obtained equalities, one derives that the system \eqref{eqn:linear_gradient_flow} admits
the following energy  dissipative law
\begin{align}
     \frac{d}{dt}\mathcal{E}(\phi)=-\left\|\sqrt{M(\phi)}\nabla\mu\right\|^2\leq0,
\end{align}
where $\|\cdot\|$ is $L^2$-morm.
Additionally, the CH model preserves local mass density. Indeed, by taking the $L^2$ inner product of equation \eqref{eqn:linear_1} with $1$, the mass conservation property can be directly obtained using integration by parts, as follow
\begin{align}
    \frac{d}{dt}\int_{\Omega}\phi d\Omega=0.
\end{align}
\par
For the Willmore regularization system, the law of dissipation energy remains valid, while also conserving local mass density.

\section{Numerical schemes on uniform temporal mesh}\label{section_3}
In this section, the schemes of uniform-time-step WSBDF2 method are built and studied. Let $\{t^n|t^n=n\tau,0\leq n\leq N\}$ be the time nodes on the interval $[0,T]$ with the uniform time step $\tau=T/N$, and denote $u^n$ as the numerical solution of $u(t^{n})$ for any function $u(t)$. Then, we define
\begin{align}\label{eqn:WBDF2}
&D_{\tau}^{\theta}u^{n+\theta}=\frac{(\theta+\frac{1}{2})u^{n+1}-2\theta u^n+(\theta-\frac{1}{2})u^{n-1}}{\tau},
\end{align}
where $\theta\in[\frac{1}{2},1]$, and $D_{\tau}^{\theta}u^{n+\theta}$ is used to approximate the value of $u'(t^{n+\theta})$.\par
In order to prove the unconditional energy stability of subsequent numerical schemes, we initially introduce the concept of G-stability as described in the classic book by Hairer \cite{hairer1996g}.\par
To simplify the presentation, we introduce a real, symmetric and positive definite matrix
\begin{align*}
     G=\begin{pmatrix}
     g_{11} & g_{12} \\
     g_{21} & g_{22} \\
\end{pmatrix}=\begin{pmatrix}
     \frac{\theta(2\theta-1)}{2}     & - \frac{(\theta+1)(2\theta-1)}{2}\\ - \frac{(\theta+1)(2\theta-1)}{2}  &  \frac{\theta(2\theta+3)}{2}
\end{pmatrix},
 \end{align*}
and define the corresponding G-norm $\|\cdot\|_G$ with $L^2$ inner $(\cdot,\cdot)$ as follow
\begin{align}\label{eqn:G_norm}
\|\mathbf{U}\|_{G}^{2}=\sum_{i=1}^{2}\sum_{j=1}^{2}g_{i,j}\left(u^{i},u^{j}\right),
\end{align}
where vector $\mathbf{U}=[u^1,u^2]^T$, with $u^1,u^2\in L^2(\Omega)$. Clearly, if $\theta=\frac{1}{2}$, the G-norm will degenerate into the $L^2$-norm. When $\theta\in(\frac{1}{2},1]$, we can easily verify that the G-norm \eqref{eqn:G_norm} is equivalent to the $L^2$-norm.
\par
Similarly, we can define another G-norm $|\cdot|_G$ as
\begin{align}
    |\mathbf{V}|_{G}^{2}=\sum_{i=1}^{2}\sum_{j=1}^{2}g_{i,j}v^{i}v^{j}
\end{align}
with vector $\mathbf{V}=[v^1,v^2]^T$, and $v^1,v^2\in\mathbb{R}$.
\par
After defining the G-norm $\|\cdot\|_G$ with $L^2$ inner $(\cdot,\cdot)$, we introduce the following lemma
\begin{lem}\label{lem:lem3.1}
    For any given sequence $\{u^n\}$, it holds
    \begin{align}\label{eqn:Gnormid}
    \tau\left(D_{\tau}^{\theta}u^{n+\theta},\theta u^{n+1}+(1-\theta)u^{n}\right)=\frac{\left\|
    \begin{bmatrix}
            u^{n+1}\\
            u^n
    \end{bmatrix}
     \right\|_G^2-\left\|
     \begin{bmatrix}
         u^{n}\\
        u^{n-1}
    \end{bmatrix}
     \right\|_G^2}{2}
     +\frac{\left\|\alpha_2u^{n+1}+\alpha_1u^n+\alpha_0u^{n-1}\right\|^2}{4}.
    \end{align}
    where $\alpha_0=\alpha_2=(\theta(2\theta-1))^{\frac{1}{2}}$ and $\alpha_1=-(\theta(2\theta-1))^{\frac{1}{2}}$.
\end{lem}
\begin{proof}
    By the definitions of G-norm and simple algebraic operations, we can
    easily complete the proof.
\end{proof}

Similar to Lemma \ref{lem:lem3.1}, for the G-norm $|\cdot|$, the following identity holds
\begin{align}\label{eqn:Gnorm_id2}
    \begin{aligned}
        &\left(\frac{2\theta+1}{2}v^{n+1}-2\theta v^n+\frac{2\theta-1}{2}v^{n-1}\right)\left(\theta v^{n+1}+(1-\theta)v^n\right)
        \\
        =&\frac{1}{2}\left|
        \begin{bmatrix}
            v^{n+1}\\
            v^n
        \end{bmatrix}
        \right|_{G}^2
        -\frac{1}{2}\left|
        \begin{bmatrix}
            v^n\\
            v^{n-1}
        \end{bmatrix}
        \right|_{ G}^2
        +\frac{\left\|\alpha_2v^{n+1}+\alpha_1 v^n+\alpha_0 v^{n-1}\right\|^2}{4}.
    \end{aligned}
\end{align}

\subsection{Linear regularization model}
For the linear regularization system, we first introduce an auxiliary variable $r(t)$, defined by
\begin{align*}
    r(t)=\sqrt{E_1(\phi)}:=\sqrt{\int_{\Omega}\gamma(\mathbf{n})\left(\frac{1}{2}|\nabla\phi|^2+\frac{1}{\epsilon^2}F(\phi)\right)d\Omega+C},
\end{align*}
where $C$ is a positive constant that guarantees the radicand is positive. Then the total free energy \eqref{eqn:total_energy} can be rewritten as
\begin{align}
    \mathcal{E}(r,\phi)=r^2-C+\frac{\beta}{2}\int_\Omega(\Delta\phi)^2d\Omega,
\end{align}
and the transformed $H^{-1}$ gradient flow is given by
\begin{subequations}\label{eqn:modified_linear_gradient_flow}
    \begin{align}
\phi_t&=\nabla\cdot(M(\phi)\nabla\mu),\\
\mu&=H(\phi)r+\beta\Delta^2\phi,\\
r_t&=\frac{1}{2}\int_{\Omega}H(\phi)\phi_td\Omega,
    \end{align}
\end{subequations}
where
\begin{align}
    H(\phi)=\frac{H_1(\phi)}{\sqrt{E_1(\phi)}}:=\frac{\frac{1}{\epsilon^2}\gamma(\mathbf{n})f(\phi)-\nabla\cdot\mathbf{m}}{\sqrt{\int_{\Omega}\gamma(\mathbf{n})\left(\frac{1}{2}|\nabla\phi|^2+\frac{1}{\epsilon^2}F(\phi)\right)d\Omega+C_0}}.
\end{align}
By taking the inner products of the above equations with $\mu$, $\phi_t$ and $2r$, we can derive
\begin{align}
   \frac{d}{dt}\mathcal{E}(r,\phi)=-\left\|\sqrt{M(\phi)}\nabla\mu\right\|^2\leq0.
\end{align}
Namely, the transformed system \eqref{eqn:modified_linear_gradient_flow} follows the new energy dissipation law.\par
We now construct a numerical scheme based on the uniform-time-step WSBDF2 method for the above system.
  \begin{itemize}
      \item {\bf{$\mathcal{U}_L$-method}}.
      Given $\phi^n$, $r^n$ and $\phi^{n-1}$, $r^{n-1}$, we update $\phi^{n+1}$, $r^{n+1}$ by solving
      \end{itemize}
\begin{subequations}\label{eqn:Uniform-WBDF2-Linear}
    \begin{align}
        &D_{\tau}^{\theta}\phi^{n+\theta}=\nabla\cdot(M^{*,n+\theta}\nabla\mu^{n+\theta}),\label{eqn:IMEX-WSBDF_1}\\
        &\begin{aligned}\mu^{n+\theta}=H^{*,n+\theta}r^{n+\theta}+\beta\Delta^2\phi^{n+\theta}+\frac{S_1}{\epsilon^2}(\phi^{n+1}-2\phi^n+\phi^{n-1})-S_2\Delta(\phi^{n+1}-2\phi^n+\phi^{n-1}),
        \end{aligned}\label{eqn:IMEX-WSBDF_2}\\
        &D_{\tau}^{\theta}r^{n+\theta}=\frac{1}{2}\int_{\Omega}H^{*,n+\theta}D_{\tau}^{\theta}\phi^{n+\theta}d\Omega,\label{eqn:IMEX-WSBDF_3}
    \end{align}
\end{subequations}
where
\begin{align}
    &\phi^{n+\theta}=\theta \phi^{n+1}+(1-\theta)\phi^{n},\quad\phi^{*,n+\theta}=(1+\theta)\phi^{n}-\theta \phi^{n-1},\nonumber\\
    &H^{*,n+\theta}=H(\phi^{*,n+\theta}),\quad M^{*,n+\theta}= M(\phi^{*,n+\theta}),\nonumber
\end{align}
and $S_i\;(i=1,2)$ are positive stabilizing parameters.\par

\begin{rem}
    In the above discretization scheme, we add two stabilization terms, $\frac{S_1}{\epsilon^2}(\phi^{n+1}-2\phi^n+\phi^{n-1})$ and $S_2\Delta(\phi^{n+1}-2\phi^n+\phi^{n-1})$, as they can eliminate the strong oscillations caused by $\gamma(\bf{n})$. These terms play a crucial role in maintaining the stability of the solution during computations. Additionally, $\frac{S_1}{\epsilon^2}(\phi^{n+1}-2\phi^n+\phi^{n-1})=\frac{S_1}{\epsilon^2}O(\tau^2)$ and $S_2\Delta(\phi^{n+1}-2\phi^n+\phi^{n-1})= S_2O(\tau^2)$, which mean they have second-order accuracy and thus do not affect the precision of the solution in practical calculations.
\end{rem}
We can observe that solving a nonlocal and coupled system for $\phi^{n+1}$ and $r^{n+1}$ within the framework of $\mathcal{U}_L$-method at each time step is a complex task. However, in practical applications, we can simplify the solving process by the following procedure.\par
   We first rewrite \eqref{eqn:IMEX-WSBDF_3} as follow
   \begin{align*}
       r^{n+1}=\frac{1}{2}\int_{\Omega}H^{*,n+\theta}\phi^{n+1}d\Omega+\tilde{g}^n,
   \end{align*}
   where
   \begin{align*}
      \tilde{g}^n=\frac{4\theta r^{n}-(2\theta-1)r^{n-1}}{2\theta+1} -\frac{1}{2}\int_{\Omega}H^{*,n+\theta}\frac{4\theta \phi^{n}-(2\theta-1)\phi^{n-1}}{2\theta+1}d\Omega.
   \end{align*}
    Then the above equations can be written as
    \begin{align}\label{eqn:equation_311}
    \begin{aligned}
        A(\phi^{n+1})-\frac{\theta}{2}\nabla\cdot(M^{*,n+\theta}\nabla(H^{*,n+\theta}))\int_{\Omega}H^{*,n+\theta}\phi^{n+1}d\Omega=g^n,
        \end{aligned}
    \end{align}
    where
    \begin{align*}
     \left\{
    \begin{aligned}
        &A=\frac{2\theta+1}{2\tau}I-\nabla\cdot\left(M^{*,n+\theta}\nabla\left(\beta\theta\Delta^2+S_1/\epsilon^2I-S_2\Delta\right)\right),\\
        &g^n=\frac{4\theta \phi^{n}-(2\theta-1)\phi^{n-1}}{2\tau}+ (\theta\tilde{g}^n+(1-\theta) r^n)\nabla\cdot(M^{*,n+\theta}\nabla(H^{*,n+\theta}))\\
        &\quad\quad -\nabla\cdot\left(M^{*,n+\theta}\nabla\left(\beta(1-\theta)\Delta^2\phi^n+S_1/\epsilon^2\phi^{*,n+1}-S_2\Delta\phi^{*,n+1}\right)\right).
        \end{aligned}
        \right.
    \end{align*}
     By applying the linear operator $A^{-1}$ to \eqref{eqn:equation_311}, then taking the $L^2$ inner product with $H^{*,n+\theta}$, we obtain
     \begin{align}\label{eqn:equation_312}
     \int_{\Omega}H^{*,n+\theta}\phi^{n+1}d\Omega=\frac{\int_{\Omega}H^{*,n+\theta}A^{-1}(g^n)d\Omega}{1-\frac{\theta}{2}\int_{\Omega}H^{*,n+\theta}A^{-1}(\nabla\cdot(M^{*,n+\theta}\nabla(H^{*,n+\theta})))d\Omega}.
     \end{align}
     It is easy to verify the term $\int_{\Omega}H^{*,n+\theta}A^{-1}(\nabla\cdot(M^{*,n+\theta}\nabla(H^{*,n+\theta})))d\Omega\geq0$ since $A^{-1}(\nabla\cdot(M^{*,n+\theta}\nabla\left(\cdot\right)))$ is a positive definite operator. Finally, we can solve $\phi^{n+1}$ from \eqref{eqn:equation_311}.\par
     To sum up, the $\mathcal{U}_L$-method \eqref{eqn:Uniform-WBDF2-Linear} can be easily implemented in the following steps:
     \begin{itemize}
     \item  Compute $A^{-1}(g^n)$ and $A^{-1}(\nabla\cdot(M^{*,n+\theta}\nabla(H^{*,n+\theta})))$ by solving two sixth-order equations;
     \end{itemize}
     \begin{itemize}
     \item  Compute $\int_{\Omega}H^{*,n+\theta}\phi^{n+1}d\Omega$ from \eqref{eqn:equation_312};
     \end{itemize}
     \begin{itemize}
     \item  Compute $\phi^{n+1}$ from the variation of \eqref{eqn:equation_311} as follow
     \end{itemize}
     \begin{align*}
         \phi^{n+1}=\frac{\theta}{2}A^{-1}(\nabla\cdot(M^{*,n+\theta}\nabla(H^{*,n+\theta})))\int_{\Omega}H^{*,n+\theta}\phi^{n+1}d\Omega+A^{-1}(g^n).
     \end{align*}
     Hence, the total cost at each time step essentially involves solving two fourth-order equations. Therefore, this method is extremely efficient and easy to implement.\par
     Next, we give the proof that the $\mathcal{U}_L$-method \eqref{eqn:UE_linear} is unconditionally energy stable, and preserves mass conservation.

\begin{theorem} \label{th:U_Energy_linear}
    The $\mathcal{U}_L$-method is unconditionally energy stable in the sense that
    \begin{align}
        \frac{1}{\tau}(\mathcal{E}_{L}^{n+1}-\mathcal{E}_{L}^{n})\leq-\left\|\sqrt{M^{*,n+\theta}}\nabla\mu^{n+\theta}\right\|^2\leq0,
    \end{align}
    where
    \begin{align} \label{eqn:UE_linear}
        \begin{aligned}
          \mathcal{E}_{L}^{n+1}=\left|
     \begin{bmatrix}
        r^{n+1}\\
        r^{n}
    \end{bmatrix}
    \right|_G^2+\frac{\beta}{2}\left\|
    \begin{bmatrix}
    \Delta\phi^{n+1}\\
    \Delta\phi^n
    \end{bmatrix}
     \right\|_G^2+\frac{S_1}{\epsilon^2}\frac{\|\phi^{n+1}-\phi^{n}\|^2}{2}+S_2\frac{\|\nabla\phi^{n+1}-\nabla\phi^{n}\|^2}{2}.
        \end{aligned}
    \end{align}
\end{theorem}
\begin{proof}
    Taking the $L^2$ inner product of \eqref{eqn:IMEX-WSBDF_1} with $\tau\mu^{n+\theta}$
    \begin{align}\label{eqn:linear_inner_product_1}
        \left((\theta+\frac{1}{2})\phi^{n+1}-2\theta \phi^n+(\theta-\frac{1}{2})\phi^{n-1},\mu^{n+\theta}\right)=-\tau\left\|\sqrt{M^{*,n+\theta}}\nabla\mu^{n+\theta}\right\|^2.
    \end{align}
    By taking the $L^2$ inner product of \eqref{eqn:IMEX-WSBDF_2} with $(\theta+\frac{1}{2})\phi^{n+1}-2\theta \phi^n+(\theta-\frac{1}{2})\phi^{n-1}$ and applying \eqref{eqn:Gnormid} along with the following identity:
        \begin{align}
            ((2\theta+1)a-4\theta b+(2\theta-1)c)(a-2b+c)=(a-b)^2-(b-c)^2+2\theta(a-2b+c)^2,
        \end{align}
        we obtain
        \begin{align}\label{eqn:linear_inner_product_2}
        \begin{aligned}
        &\left(\mu^{n+\theta},(\theta+\frac{1}{2})\phi^{n+1}-2\theta \phi^n+(\theta-\frac{1}{2})\phi^{n-1}\right)\\
        =&r^{n+\theta}\left(H^{*,n+\theta},(\theta+\frac{1}{2})\phi^{n+1}-2\theta \phi^n+(\theta-\frac{1}{2})\phi^{n-1}\right)\\
        &+\beta\left(\frac{1}{2}\left\|
    \begin{bmatrix}
    \Delta\phi^{n+1}\\
    \Delta\phi^n
    \end{bmatrix}
     \right\|_G^2-\frac{1}{2}\left\|
     \begin{bmatrix}
        \Delta\phi^{n}\\
        \Delta\phi^{n-1}
    \end{bmatrix}
     \right\|_G^2
     +\frac{\left\|\alpha_2\Delta\phi^{n+1}+\alpha_1\Delta\phi^n+\alpha_0\Delta\phi^{n-1}\right\|^2}{4}\right)\\
     &+\frac{S_1}{\epsilon^2}\left(\frac{\|\phi^{n+1}-\phi^{n}\|^2}{2}-\frac{\|\phi^{n}-\phi^{n-1}\|^2}{2}+\theta\|\phi^{n+1}-2\phi^{n}+\phi^{n-1}\|^2\right)\\
     &+S_2\left(\frac{\|\nabla\phi^{n+1}-\nabla\phi^{n}\|^2}{2}-\frac{\|\nabla\phi^{n}-\nabla\phi^{n-1}\|^2}{2}+\theta\|\nabla\phi^{n+1}-2\nabla\phi^{n}+\nabla\phi^{n-1}\|^2\right).
           \end{aligned}
        \end{align}
By multiplying \eqref{eqn:IMEX-WSBDF_3} with $r^{n+\theta}$ and using \eqref{eqn:Gnorm_id2}, we have
\begin{align}\label{eqn:linear_inner_product_3}
\begin{aligned}
    &\frac{1}{2}\left|
    \begin{bmatrix}
    r^{n+1}\\
    r^n
    \end{bmatrix}
     \right|_G^2-\frac{1}{2}\left|
     \begin{bmatrix}
        r^{n}\\
        r^{n-1}
    \end{bmatrix}
     \right|_G^2
     +\frac{\left|\alpha_2r^{n+1}+\alpha_1r^n+\alpha_0r^{n-1}\right|^2}{4}\\
    =&\frac{1}{2}r^{n+\theta}\int_{\Omega}H^{*,n+\theta}\left((\theta+\frac{1}{2})\phi^{n+1}-2\theta \phi^n+(\theta-\frac{1}{2})\phi^{n-1}\right)d\Omega.
    \end{aligned}
\end{align}
Combining the equations \eqref{eqn:linear_inner_product_1},\eqref{eqn:linear_inner_product_2} and \eqref{eqn:linear_inner_product_3}, we derive
\begin{align}
\begin{aligned}
    &\left|
    \begin{bmatrix}
    r^{n+1}\\
    r^n
    \end{bmatrix}
     \right|_G^2-\left|
     \begin{bmatrix}
        r^{n}\\
        r^{n-1}
    \end{bmatrix}
     \right|_G^2+\frac{\beta}{2}\left(\left\|
    \begin{bmatrix}
    \Delta\phi^{n+1}\\
    \Delta\phi^n
    \end{bmatrix}
     \right\|_G^2-\left\|
     \begin{bmatrix}
        \Delta\phi^{n}\\
        \Delta\phi^{n-1}
    \end{bmatrix}
     \right\|_G^2\right)\\
     &+\frac{S_1}{\epsilon^2}\left(\frac{\|\phi^{n+1}-\phi^{n}\|^2}{2}-\frac{\|\phi^{n}-\phi^{n-1}\|^2}{2}\right)+S_2\left(\frac{\|\nabla\phi^{n+1}-\nabla\phi^{n}\|^2}{2}-\frac{\|\nabla\phi^{n}-\nabla\phi^{n-1}\|^2}{2}\right)\\
     &+\frac{\left|\alpha_2r^{n+1}+\alpha_1r^n+\alpha_0r^{n-1}\right|^2}{2}+\frac{\beta}{2}\left(\frac{\left\|\alpha_2\Delta\phi^{n+1}+\alpha_1\Delta\phi^n+\alpha_0\Delta\phi^{n-1}\right\|^2}{2}\right)\\
     &+\frac{\theta S_1}{\epsilon^2}\|\phi^{n+1}-2\phi^{n}+\phi^{n-1}\|^2+\theta S_2|\nabla\phi^{n+1}-2\nabla\phi^{n}+\nabla\phi^{n-1}\|^2\\
     &=-\tau\|\sqrt{M^{*,n+\theta}}\nabla\mu^{n+\theta}\|^2.
     \end{aligned}
\end{align}
Finally, we obtain the desired result by omitting the positive terms.
\end{proof}

 \begin{theorem}\label{th:U_Mass_linear}
         The solution of the $\mathcal{U}_L$-method \eqref{eqn:Uniform-WBDF2-Linear} satisfies the mass conservation.
     \end{theorem}
     \begin{proof}
         By taking the inner product of equation \eqref{eqn:IMEX-WSBDF_1} with 1, we can immediately obtain
         \begin{align*}
         \left(\frac{(\theta+\frac{1}{2})u^{n+1}-2\theta u^n+(\theta-\frac{1}{2})u^{n-1}}{\tau},1\right)=\nabla\cdot(M^{*,n+\theta}\nabla\mu^{n+\theta})=0,
           \end{align*}
          namely
          \begin{align*}
              \int_{\Omega}\phi^{n+1}d\Omega=\frac{4\theta}{2\theta+1}\int_{\Omega}\phi^{n}d\Omega-\frac{2\theta-1}{2\theta+1}\int_{\Omega}\phi^{n-1}d\Omega.
          \end{align*}
         By applying mathematical induction with the initial condition $\int_{\Omega}\phi^{1}d\Omega=\int_{\Omega}\phi^{0}d\Omega$, we can conclude
         \begin{align}
           \int_{\Omega}\phi^{n+1}d\Omega=\int_{\Omega}\phi^{n}d\Omega=\cdots=\int_{\Omega}\phi^{1}d\Omega=\int_{\Omega}\phi^{0}d\Omega.
         \end{align}
         The proof is complete.
     \end{proof}

\begin{rem}
    Since the BDF2 method is a two-step method, we use the BDF1 method, also known as the backward Euler method, to compute $\phi^1$. For this method, we can easily obtain that $\int_{\Omega}\phi^{1}d\Omega=\int_{\Omega}\phi^{0}d\Omega$, which means that mass conservation is maintained in the first step of the computation.
\end{rem}
\subsection{Willmore regularization model}
we consider the Willmore regularization model in this subsection. Similar to the case of linear regularization model, we define an auxiliary variable as follow
\begin{align*}
    r(t)=\sqrt{E_2(\phi)}:=\sqrt{\int_{\Omega}\left(\gamma(\mathbf{n})\left(\frac{1}{2}|\nabla\phi|^2+\frac{1}{\epsilon^2}F(\phi)\right)+\frac{\beta}{2}\left(\Delta\phi-\frac{1}{\epsilon^2}f(\phi)\right)^2\right)d\Omega+C},
\end{align*}
where $C$ is a positive constant that ensures the radicand is positive. Then the total free energy can \eqref{eqn:total_energy} be rewritten as
\begin{align}
    \mathcal{E}(r,\phi)=r^2-C,
\end{align}
and the transformed $H^{-1}$ gradient flow is given by
\begin{subequations}\label{eqn:modified_willmore_gradient_flow}
    \begin{align}
\phi_t&=\nabla\cdot(M(\phi)\nabla\mu),\\
\mu&=Z(\phi)r,\\
r_t&=\frac{1}{2}\int_{\Omega}Z(\phi)\phi_td\Omega,
    \end{align}
\end{subequations}
where
\begin{align}
    Z(\phi)=\frac{Z_1(\phi)}{\sqrt{E_2(\phi)}}:=\frac{\frac{1}{\epsilon^2}\gamma(\mathbf{n})f(\phi)-\nabla\cdot\mathbf{m}+\beta\left(\Delta\phi-\frac{1}{\epsilon^2}f(\phi)\right)\left(\Delta\phi-\frac{1}{\epsilon^2}f'(\phi)\right)}{\sqrt{\int_{\Omega}\left(\gamma(\mathbf{n})\left(\frac{1}{2}|\nabla\phi|^2+\frac{1}{\epsilon^2}F(\phi)\right)+\frac{\beta}{2}\left(\Delta\phi-\frac{1}{\epsilon^2}f(\phi)\right)^2\right)d\Omega+C}}.
\end{align}
Similar as the linear regularization model, by taking the inner products of the above equations with $\mu$, $\phi_t$ and $2r$, we can derive
\begin{align}
   \frac{d}{dt}\mathcal{E}(r,\phi)=-\left\|\sqrt{M(\phi)}\nabla\mu\right\|^2\leq0.
\end{align}
Namely, the transformed linear regularization system \eqref{eqn:modified_willmore_gradient_flow} follows the modified  energy dissipation law.\par
We now construct a numerical scheme based on the uniform-time-step WSBDF2 method for the system \eqref{eqn:modified_willmore_gradient_flow}.
\begin{itemize}
     \item {\bf{$\mathcal{U}_W$-method}.}
    Given $\phi^n$, $r^n$ and $\phi^{n-1}$, $r^{n-1}$, we update $\phi^{n+1}$, $r^{n+1}$ by solving
\end{itemize}
\begin{subequations}\label{eqn:Uniform-WBDF2-willmore}
    \begin{align}
        &D_{\tau}^{\theta}\phi^{n+\theta}=\nabla\cdot(M^{*,n+\theta}\nabla\mu^{n+\theta}),\label{eqn:IMEX-WSBDF_willmore_1}\\
        &\begin{aligned}
            \mu^{n+\theta}=&Z^{*,n+\theta}r^{n+\theta}+\frac{S_1}{\epsilon^2}(\phi^{n+1}-2\phi^n+\phi^{n-1})\\&-S_2\Delta(\phi^{n+1}-2\phi^n+\phi^{n-1})+S_3\Delta^2(\phi^{n+1}-2\phi^n+\phi^{n-1}),
        \end{aligned}\label{eqn:IMEX-WSBDF_willmore_2}\\
        &D_{\tau}^{\theta}r^{n+\theta}=\frac{1}{2}\int_{\Omega}Z^{*,n+\theta}D_{\tau}^{\theta}\phi^{n+\theta}d\Omega,\label{eqn:IMEX-WSBDF_willmore_3}
    \end{align}
\end{subequations}
where $Z^{*,n+\theta}=Z(\phi^{*,n+\theta})$ and $S_i\;(i=1,2,3)$ are positive stabilizing parameters.
\par
The following two theorems demonstrate that the structure-preserving properties still hold for the Willmore regularization system. Additionally, since the proof processes are similar to Theorem \ref{th:U_Energy_linear} and Theorem \ref{th:U_Mass_linear}, the proofs are omitted here.
\begin{theorem}
     The $\mathcal{U}_W$-method is unconditionally energy stable in the sense that
    \begin{align}
        \frac{1}{\tau}(\mathcal{E}_{W}^{n+1}-\mathcal{E}_{W}^{n})\leq\textit{\textit{}}-\left\|\sqrt{M^{*,n+\theta}}\nabla\mu^{n+\theta}\right\|^2\leq0,
    \end{align}
    where
    \begin{align}
        \begin{aligned}
          \mathcal{E}_{W}^{n+1}=&\left|
     \begin{bmatrix}
        r^{n+1}\\
        r^{n}
    \end{bmatrix}
    \right|_G^2+\frac{S_1}{\epsilon^2}\frac{\|\phi^{n+1}-\phi^{n}\|^2}{2}+S_2\frac{\|\nabla\phi^{n+1}-\nabla\phi^{n}\|^2}{2}+S_3\frac{\|\Delta\phi^{n+1}-\Delta\phi^{n}\|^2}{2}.
        \end{aligned}
    \end{align}
\end{theorem}
\begin{theorem}
     The solution of the $\mathcal{U}_W$-method \eqref{eqn:Uniform-WBDF2-willmore} satisfies the mass conservation.
\end{theorem}

\section{Numerical scheme on nonuniform temporal mesh}\label{section_4}
In this section, we introduce the varia-time-step WSBDF2 method for the anisotropic model and prove its property of energy stability.\par
Choose the time levels $0={t_0}<{t_1}<{t_2}<\cdots <{t_N}=T$ with the $\tau_n:={t_n}-{t_{n-1}}$ for $1\leq n \leq N$. Let $\gamma_{n+1}=\tau_{n+1}/\tau_{n}$ be the adjacent time step ratios and set $\tau=max\{\tau_n,1\leq n \leq N\}$ as the maximum time step size. For any sequence $\{u^n\}_{n=0}^{N}$, we denote $\nabla_{\tau}u^{n}:=u^{n}-u^{n-1}$ for $1\leq n\leq N$ and define
\begin{align}\label{eqn:VST_WBDF2}
    \Tilde{D}_{\tau}^{\theta}u^{n+\theta}=\frac{1+2\theta\gamma_{n+1}}{\tau_{n+1}\left(1+\gamma_{n+1}\right)}\nabla_{\tau}u^{n+1}+\frac{\left(1-2\theta\right)\gamma_{n+1}^{2}}{\tau_{n+1}\left(1+\gamma_{n+1}\right)}\nabla_{\tau}u^{n}.
\end{align}
  However, unfortunately, on the  nonuniform temporal mesh, we cannot derive a result similar to Lemma \ref{lem:lem3.1}.
This is because, if we follow the approach similar to that used to introduce the G-norm for achieving \eqref{eqn:Gnormid}, inspired by the uniform G-stability idea, we cannot obtain a real, symmetric, and positive definite matrix, and thus we do not achieve energy dissipation.

In order to prove unconditional energy stability on nonuniform temporal mesh, we need to introduce the  Theorem 3.1 in \cite{chen2021weighted} or Lemma 2.1 in \cite{liao2022mesh}.

\begin{lem}
Let the adjacent step ratios $\gamma_{n+1}$ satisfy $0\leq\gamma_{n+1}\leq\gamma_{*}$, then it holds that
\begin{align}\label{eqn:inequality}
    \Tilde{D}_{2}^{n+\theta}\phi^{n+\theta}(\phi^{n+1}-\phi^{n})\geq R_{\theta}^{n+1}-R_{\theta}^{n}+R(\gamma_{n+1},\gamma_{n+2})\frac{\|\phi^{n+1}-\phi^{n}\|^2}{2\tau_{n+1}},
\end{align}
where $\gamma_{*}$ is the positive root of equation
\begin{align}
    (1-2\theta)^2\gamma_{*}^{3}-4\theta^2\gamma_{*}^{2}-4\theta\gamma_{*}-1=0,\nonumber
\end{align}
and
\begin{align*}
  R_{\theta}^{n}=\frac{(2\theta-1)\gamma_{n+1}^{3/2}}{2(1+\gamma_{n+1})}\frac{\|\phi^{n}-\phi^{n-1}\|^2}{\tau_{n}},\quad R(x,y):=\frac{2(1+2\theta x)+(1-2\theta)x^{3/2}}{1+x}-\frac{(2\theta-1)y^{3/2}}{1+y}.
\end{align*}
\end{lem}
\par
It is easy to verify that $R(x,y)$ is increasing in  interval $(0,1)$ and decreasing in the interval with respect to $x$. Additionally, $R(x, y)$ is decreasing with respect to $y$. Then we have
\begin{align*}
    R(x,y)\geq min\{R(0,\gamma_{*}),R(\gamma_{*},\gamma_{*})\}=\frac{(1-2\theta)^2\gamma_{*}^{3}-4\theta^2\gamma_{*}^{2}-4\theta\gamma_{*}-1}{1+\gamma_{*}}=0,\quad 0\leq x,y\leq\gamma_{*}.
\end{align*}
Moreover, we can easily check that the root $\gamma_{*}$ is decreasing for $\theta\in[\frac{1}{2},1]$. In particular, $\gamma_{*}=4.8645365123$ if $\theta=1$ and $\gamma_{*}\rightarrow\infty$ if $\theta=\frac{1}{2}$.

\subsection{Linear regularization model}
Different from the SAV approach on uniform temporal mesh, we reintroduce an auxiliary variable, denoted by
\begin{align*}
    u(t)=\sqrt{\widetilde{E}_1(\phi)}:=\sqrt{\int_{\Omega}\gamma(\mathbf{n})\left(\frac{1}{2}|\nabla\phi|^2+\frac{1}{\epsilon^2}F(\phi)\right)-\frac{\lambda_1}{2\epsilon^2}\left|\phi\right|^2 -\frac{\lambda_2}{2}\left|\nabla\phi\right|^2 d\Omega+C_0},
\end{align*}
where $C_0$ is a positive constant selected to ensure the value $\widetilde{E}_1$ is positive. Now, the total free energy \eqref{eqn:total_energy} can be rewritten as
\begin{align}\label{eqn:Vlinear_total_energy}
    \mathcal{E}(u,\phi)=u^2+\frac{\beta}{2}\int_\Omega(\Delta\phi)^2d\Omega+\int_\Omega\left(\frac{\lambda_1}{2\epsilon^2}\left|\phi\right|^2 +\frac{\lambda_2}{2}\left|\nabla\phi\right|^2\right)d\Omega-C_0.
\end{align}
Then, we can obtain a modified gradient flow
\begin{subequations}\label{eqn:Vmodified_linear_gradient_flow}
    \begin{align}
\phi_t&=\nabla\cdot(M(\phi)\nabla\mu),\\
\mu&=\frac{u}{\sqrt{\widetilde{E}_1(\phi)}}V(\xi)\mathcal{H(\phi)}+\beta\Delta^2\phi+\frac{\lambda_1}{\epsilon^2}\phi-\lambda_2\Delta\phi,\\
u_t&=\frac{V(\xi)}{2\sqrt{\widetilde{E}_1(\phi)}}\int_{\Omega}\mathcal{H}(\phi)\phi_t d\Omega,
    \end{align}
\end{subequations}
where
\begin{align}
    \mathcal{H}(\phi)=\frac{1}{\epsilon^2}\gamma(\mathbf{n})f(\phi)-\nabla\cdot\mathbf{m}-\frac{\lambda_1}{\epsilon^2}\phi+\lambda_2\Delta\phi,
\end{align}
and $V(\xi)\in C^2(\mathbb{R})$ is a real positive function with $V(1)\equiv 1$. Taking $L^2$ inner product of equations in \eqref{eqn:Vmodified_linear_gradient_flow} with $\mu$, $\phi_t$ and $2u$ respectively, we can obtain
\begin{align}
    \frac{d}{dt}\mathcal{E}(u,\phi)=-\left\|\sqrt{M(\phi)}\nabla\mu\right\|^2\leq 0,
\end{align}
which means that the modified system \eqref{eqn:Vmodified_linear_gradient_flow} satisfies the unconditional energy stability.\par
In what follows, we construct a numerical scheme with variable-time-step method to discretize the above linear regularized system.
\begin{itemize}
    \item {\bf{$\mathcal{V}_L$-method}.} Given $\phi^n$, $r^n$ and $\phi^{n-1}$, update $\phi^{n+1}$, $r^{n+1}$ by solving
\end{itemize}
\begin{subequations}\label{eqn:VTS-WBDF2-Linear}
    \begin{align}
        &\Tilde{D}_{2}^{n+\theta}\phi^{n+\theta}=\nabla\cdot(M^{*,n+\theta}\nabla\mu^{n+\theta}),\label{eqn:IMEX-VST_WBDF_linear_1}\\
        &\mu^{n+\theta}=\frac{u^{n+1}}{\sqrt{\widetilde{E}_1^{n}}}V(\xi^{n+1})\mathcal{H}^{*,n+\theta}+\beta\Delta^2\phi^{n+\theta}+\frac{\lambda_1}{\epsilon^2}\phi^{n+\theta}-\lambda_2\Delta\phi^{n+\theta},\label{eqn:IMEX-VST_WBDF_linear_2}\\
        &\frac{u^{n+1}-u^n}{\tau_{n+1}}=\frac{V(\xi^{n+1})}{2\sqrt{\widetilde{E}_1^{n}}}\int_{\Omega}\mathcal{H}^{*,n+\theta}\frac{\phi^{n+1}-\phi^n}{\tau_{n+1}}d\Omega,\label{eqn:IMEX-VST_WBDF_linear_3}
    \end{align}
\end{subequations}
where
\begin{align*}
    \xi^{n+1}=\frac{u^{n+1}}{\sqrt{\widetilde{E}_1^{n}}},\quad\phi^{n+\theta}=\theta\phi^{n+1}+(1-\theta)\phi^{n},\quad\phi^{*,n+\theta}=(1+\theta\gamma_{n+1})\phi^{n}-\theta\gamma_{n+1} \phi^{n-1},
\end{align*}
$\lambda_i$ (i=1,2) are positive constants and $\lim\limits_{\xi\to 1}\frac{V(\xi)-1}{1-\xi}=1$ with $V(1)=1$. Considering \eqref{eqn:IMEX-VST_WBDF_linear_3}, it is clear that the numerical approximations for $u^{n+1}$ is  a first-order approximation of $u(t)$, and thus, $\xi^{n+1}$ have only first-order accuracy of $1$ in time, which means that
\begin{align}
    u^{n+1}=u(t^{n+1})+O(\tau_{n+1}),\quad \xi^{n+1}=1+O(\tau_{n+1}).
\end{align}
Therefore, it may initially appear that the scheme we construct is first-order accuracy, but in reality, the phase function $\phi$  in the original gradient flow \eqref{eqn:linear_gradient_flow} is not directly influenced by $r(t)$ and $\xi$; instead, it is controlled by $\frac{u(t)}{\sqrt{\widetilde{E}_{1}}}V(\xi)$. Hence, the scheme we develop is, in fact, second-order accuracy. Due to the limitations in the length of this paper, we plan to conduct an error analysis of this scheme in our future work.\par
Next, we will show how to efficiently solve the $\mathcal{U}_L$-method \eqref{eqn:VTS-WBDF2-Linear}. It follows from
\eqref{eqn:IMEX-VST_WBDF_linear_1} that
\begin{align}
    \begin{aligned}
        A(\phi^{n+1})-\xi^{n+1}V(\xi^{n+1})\nabla\cdot(M^{*,n+\theta}\nabla(\mathcal{H}^{*,n+\theta}))=g^n,
    \end{aligned}
\end{align}
where

\begin{align*}
\left\{
    \begin{aligned}
        A=&\frac{1+2\theta\gamma_{n+1}}{\tau_{n+1}\left(1+\gamma_{n+1}\right)}I-\theta\nabla\cdot\left(M^{*,n+\theta}\nabla\left(\beta\Delta^2+\frac{\lambda_1}{\epsilon^2}I-\lambda_2\Delta\right)\right),\\
        g^n=&\frac{1}{\tau_{n+1}}\left((1+(2\theta-1)\gamma_{n+1})\phi^{n}-\frac{(2\theta-1)\gamma_{n+1}^{2}}{1+\gamma_{n+1}}\phi^{n-1}\right)\\
        &+(1-\theta)\nabla\cdot\left(M^{*,n+\theta}\nabla\left(\beta\Delta^2+\frac{\lambda_1}{\epsilon^2}I-\lambda_2\Delta\right)\right).
    \end{aligned}
\right.
\end{align*}
Then we denote
\begin{align}\label{eqn:equation_45}
    \phi^{n+1}=\phi_1^{n+1}+\xi^{n+1}V(\xi^{n+1})\phi_2^{n+1}
\end{align}
with $\phi_1^{n+1}$ and $\phi_2^{n+1}$ being solved respectively by
    \begin{align}\label{eqn:equation_46}
    A(\phi_1^{n+1})=g^n,\quad
    A(\phi_2^{n+1})=\nabla\cdot(M^{*,n+\theta}\nabla(\mathcal{H}^{*,n+\theta})).
\end{align}
Next, we numerically solve for $\xi^{n+1}$ by substituting \eqref{eqn:equation_45} into \eqref{eqn:IMEX-VST_WBDF_linear_3}, resulting in
\begin{align}\label{eqn:equation_47}
    \begin{aligned}
      \xi^{n+1}\sqrt{\widetilde{E}_1^n}-u^n-\frac{V(\xi^{n+1})}{2\sqrt{\widetilde{E}_1^n}}\left[\xi^{n+1}V(\xi^{n+1})\left(\mathcal{H}^{*,n+\theta},\phi_2^{n+1}\right)+\left(\mathcal{H}^{*,n+\theta},\phi_1^{n+1}-\phi^n\right)\right]=0.
    \end{aligned}
\end{align}
Denoting the left side of the above equation by $W(\xi^{n+1})$, combining $V(1)=1$ and the turth $V'(1)=-1$, we obtain the following identities
\begin{align}
    \begin{aligned}
        &W(1)=\sqrt{\widetilde{E}_1^n}-u^n-\frac{\left(\mathcal{H}^{*,n+\theta},\phi_1^{n+1}+\phi_2^{n+1}-\phi^n\right)}{2\sqrt{\widetilde{E}_1^n}}\sim\mathrm{O}(\tau_{n+1}),\\
        &W'(1)=\sqrt{\widetilde{E}_1^n}+\frac{\left(\mathcal{H}^{*,n+\theta},\phi_1^{n+1}+\phi_2^{n+1}-\phi^n\right)}{2\sqrt{\widetilde{E}_1^n}}\sim\sqrt{\widetilde{E}_1^n}+\mathrm{O}(\tau_{n+1}).
    \end{aligned}
\end{align}
Therefore, $\xi^{n+1}$ can be efficiently computed by solving the nonlinear algebraic equation \eqref{eqn:equation_47} using Newton's iteration, starting with $\xi^0=1$ and a sufficiently large $C$, since we only require first-order accuracy for $\xi^{n+1}$ in time within the {\bf{$\mathcal{V}_L$}}-method \eqref{eqn:VTS-WBDF2-Linear}.
To sum up, the $\mathcal{V}_L$-method \eqref{eqn:VTS-WBDF2-Linear} can be easily implemented in the following steps:
     \begin{itemize}
     \item  Compute $\phi_1^{n+1}$ and $\phi_2^{n+1}$ form \eqref{eqn:equation_46};
     \end{itemize}
     \begin{itemize}
     \item  Compute $\xi^{n+1}$ from \eqref{eqn:equation_47}, and then $\phi^{n+1}$ can be obtain by using \eqref{eqn:equation_45}.
     \end{itemize}
     Hence, this method is extremely efficient and easy to implement.

Then, we give the proof that the $\mathcal{V}_L$-method \eqref{eqn:VE_linear} is unconditionally energy stable, and
preserves mass conservation
\begin{theorem}\label{th:V_Energy_linear}
    For $\theta\in[\frac{1}{2},1]$ and $0\leq\gamma_{n+1}\leq\gamma_{*}$, with $M(\phi)$ in \eqref{eqn:linear_gradient_flow} is a positive constant or a time-dependent positive function, the $\mathcal{V}_L$-method is unconditionally energy stable in the sense that
    \begin{align}
        \tilde{\mathcal{E}}_{L}^{n+1}-\tilde{\mathcal{E}}_{L}^{n}\leq 0,
    \end{align}
    where the modified energy is defined by
    \begin{align} \label{eqn:VE_linear}
      \begin{aligned}\tilde{\mathcal{E}}_{L}^{n+1}=&\frac{(2\theta-1)\gamma_{n+2}^{3/2}}{2(1+\gamma_{n+2})}\frac{\|\nabla^{-1}(\phi^{n+1}-\phi^{n})\|^2}{M\tau_{n+1}}+\frac{\beta}{2}\|\Delta\phi^{n+1}\|^2\\
      &+\frac{\lambda_1}{2\epsilon^2}\|\phi^{n+1}\|^2+\frac{\lambda_2}{2}\|\nabla\phi^{n+1}\|^2+|u^{n+1}|^2.
      \end{aligned}
    \end{align}
\end{theorem}
\begin{proof}
    Setting $M^{*,n+\theta}=M$ in \eqref{eqn:IMEX-VST_WBDF_linear_1} and taking the inner products of \eqref{eqn:IMEX-VST_WBDF_linear_1} and \eqref{eqn:IMEX-VST_WBDF_linear_2} with $-\frac{1}{M}\Delta^{-1}(\phi^{n+1}-\phi^n)$ and $\phi^{n+1}-\phi^n$ respectively
        \begin{align}
        \left(\Tilde{D}_{2}^{n+\theta}\phi^{n+\theta},-\frac{1}{M}\Delta^{-1}(\phi^{n+1}-\phi^n)\right)=-\left(\mu^{n+\theta},\phi^{n+1}-\phi^n\right),
        \end{align}
        and
        \begin{align}
              \left(\mu^{n+\theta},\phi^{n+1}-\phi^n\right)=&\frac{u^{n+1}}{\sqrt{\widetilde{E}_1^{n}}}V(\xi^{n+1})\left(\mathcal{H}^{*,n+\theta},\phi^{n+1}-\phi^n\right)+\beta\left(\Delta^2\phi^{n+\theta},\phi^{n+1}-\phi^n\right)\nonumber\\
            &+\frac{\lambda_1}{\epsilon^2}\left(\phi^{n+\theta},\phi^{n+1}-\phi^n\right)-\lambda_2\left(\Delta\phi^{n+\theta},\phi^{n+1}-\phi^n\right).
        \end{align}
        Then, by multiplying \eqref{eqn:IMEX-VST_WBDF_linear_3} with $2u^{n+1}$, we obtain
        \begin{align}
            \frac{2u^{n+1}(u^{n+1}-u^{n})}{\tau_{n+1}}=\frac{u^{n+1}V(\xi^{n+1})}{\sqrt{\widetilde{E}_1^{n}}}\int_{\Omega}\mathcal{H}^{*,n+\theta}\frac{\phi^{n+1}-\phi^n}{\tau_{n+1}}d\Omega.
        \end{align}
        It follows from the above equations that
        \begin{align}
        \begin{aligned}
            &\left(\Tilde{D}_{2}^{n+\theta}\phi^{n+\theta},-\frac{1}{M}\Delta^{-1}(\phi^{n+1}-\phi^n)\right)+\beta\left(\Delta^2\phi^{n+\theta},\phi^{n+1}-\phi^n\right)+2u^{n+1}(u^{n+1}-u^{n})\\
            &+\frac{\lambda_1}{\epsilon^2}\left(\phi^{n+\theta},\phi^{n+1}-\phi^n\right)+\lambda_2\left(\nabla\phi^{n+\theta},\nabla(\phi^{n+1}-\phi^n)\right)=0.
            \end{aligned}
        \end{align}
        By applying the inequality \eqref{eqn:inequality} and the identities
        \begin{subequations}
            \begin{align}
              &2\left(\theta a^{k+1}+(1-\theta)a^k\right)\left(a^{k+1}-a^k\right)=|a^{k+1}|^2-|a^{k}|^2+(2\theta-1)|a^{k+1}-a^{k}|^2,\\
              &2a^{k+1}\left(a^{k+1}-a^k\right)=|a^{k+1}|^2-|a^{k}|^2+|a^{k+1}-a^{k}|^2,
            \end{align}
        \end{subequations}
        we have
        \begin{align}
        \begin{aligned}
            &\frac{(2\theta-1)\gamma_{n+2}^{3/2}}{2(1+\gamma_{n+2})}\frac{\|\nabla^{-1}\phi^{n+1}-\nabla^{-1}\phi^{n}\|^2}{M\tau_{n+1}}-\frac{(2\theta-1)\gamma_{n+1}^{3/2}}{2(1+\gamma_{n+1})}\frac{\|\nabla^{-1}\phi^{n}-\nabla^{-1}\phi^{n-1}\|^2}{M\tau_{n}}+\frac{\beta}{2}\|\Delta\phi^{n+1}\|^2\\
            &-\frac{\beta}{2}\|\Delta\phi^{n}\|^2+\frac{\lambda_1}{2\epsilon^2}\|\phi^{n+1}\|^2-\frac{\lambda_1}{2\epsilon^2}\|\phi^{n}\|^2+\frac{\lambda_2}{2}\|\nabla\phi^{n+1}\|^2-\frac{\lambda_2}{2}\|\nabla\phi^{n}\|^2\\
            &+|u^{n+1}|^2-|u^{n}|^2+R(\gamma_{n+1},\gamma_{n+2})\frac{\|\phi^{n+1}-\phi^{n}\|^2}{2M\tau_{n+1}}+\frac{\beta(2\theta-1)}{2}\|\Delta\phi^{k}-\Delta\phi^{k-1}\|^2\\
            &+\frac{\lambda_1(2\theta-1)}{2\epsilon^2}\|\phi^{k}-\phi^{k-1}\|^2+\frac{\lambda_2(2\theta-1)}{2}\|\nabla\phi^{k}-\nabla\phi^{k-1}\|^2+|u^{k+1}-u^{k}|^2\leq0.
            \end{aligned}
        \end{align}
        Finally, we obtain the desired result by omitting the positive terms.
\end{proof}

\begin{theorem}\label{th:V_Mass_linear}
         The solution of the $\mathcal{V}_L$-method \eqref{eqn:VTS-WBDF2-Linear} satisfies the mass conservation.
     \end{theorem}
     \begin{proof}
         By taking the inner product of equation \eqref{eqn:IMEX-VST_WBDF_linear_1} with 1, we can immediately obtain
         \begin{align*}
         \left(\frac{1+2\theta\gamma_{n+1}}{\tau_{n+1}\left(1+\gamma_{n+1}\right)}\nabla_{\tau}\phi^{n+1}+\frac{\left(1-2\theta\right)\gamma_{n+1}^{2}}{\tau_{n+1}\left(1+\gamma_{n+1}\right)}\nabla_{\tau}\phi^{n},1\right)=\nabla\cdot(M^{*,n+\theta}\nabla\mu^{n+\theta})=0,
           \end{align*}
          namely

          \begin{align*}
              \int_{\Omega}\phi^{n+1}d\Omega=\left(1-\frac{(1-2\theta)\gamma_{n}^{2}}{1+2\theta\gamma_{n}}\right)\int_{\Omega}\phi^{n}d\Omega-\frac{(1-2\theta)\gamma_{n}^{2}}{1+2\theta\gamma_{n}}\int_{\Omega}\phi^{n-1}d\Omega.
          \end{align*}
         By applying mathematical induction with the initial condition $\int_{\Omega}\phi^{1}d\Omega=\int_{\Omega}\phi^{0}d\Omega$, we can conclude
         \begin{align}
           \int_{\Omega}\phi^{n+1}d\Omega=\int_{\Omega}\phi^{n}d\Omega=\cdots=\int_{\Omega}\phi^{1}d\Omega=\int_{\Omega}\phi^{0}d\Omega.
         \end{align}
         The proof is complete.
     \end{proof}

\subsection{ Willmore regularization model}For the Willmore regularization model on the nonuniform temporal mesh, we need to redefine the auxiliary variable $u(t)$ as follows
\begin{align*}   u(t)=\sqrt{\widetilde{E}_2(\phi)}:=\sqrt{\int_{\Omega}\left(g(\phi)-\frac{\lambda_1}{2\epsilon^2}\left|\phi\right|^2 -\frac{\lambda_2}{2}\left|\nabla\phi\right|^2-\frac{\lambda_3}{2}|\Delta\phi|^2\right)d\Omega+C_0},
\end{align*}
where $C_0$ is a positive constant to ensure $\widetilde{E}_2(\phi)>0$ and
\begin{align*}
g(\phi)=\gamma(\mathbf{n})\left(\frac{1}{2}|\nabla\phi|^2+\frac{1}{\epsilon^2}F(\phi)\right)+\frac{\beta}{2}\left(\Delta\phi-\frac{1}{\epsilon^2}f(\phi)\right)^2.
\end{align*}
Thus, a new modified energy for the Willmore regularization model can be given by
\begin{align}\label{eqn:Vwillmore_total_energy}
    \mathcal{E}(u,\phi)=u^2+\int_\Omega\left(\frac{\lambda_1}{2\epsilon^2}\left|\phi\right|^2 +\frac{\lambda_2}{2}\left|\nabla\phi\right|^2+\frac{\lambda_3}{2}|\Delta\phi|^2\right)d\Omega-C_0.
\end{align}
Then, we can obtain the corresponding gradient flow
\begin{subequations}\label{eqn:Vmodified_willmore_gradient_flow}
\begin{align}
\phi_t&=\nabla\cdot(M(\phi)\nabla\mu),\\
\mu&=\frac{u}{\sqrt{\widetilde{E}_2(\phi)}}V(\xi)\mathcal{Z}(\phi)+\frac{\lambda_1}{\epsilon^2}\phi-\lambda_2\Delta\phi+\lambda_3\Delta^2\phi,\\
u_t&=\frac{V(\xi)}{2\sqrt{\widetilde{E}_2(\phi)}}\int_{\Omega}\mathcal{Z}(\phi)\phi_t d\Omega,
\end{align}
\end{subequations}
where
\begin{align}
    \mathcal{Z}(\phi)=\frac{1}{\epsilon^2}\gamma(\mathbf{n})f(\phi)-\nabla\cdot\mathbf{m}+\beta\left(\Delta\phi-\frac{1}{\epsilon^2}f(\phi)\right)\left(\Delta\phi-\frac{1}{\epsilon^2}f'(\phi)\right)-\frac{\lambda_1}{\epsilon^2}\phi+\lambda_2\Delta\phi-\lambda_3\Delta^2\phi,
\end{align}
and $V(\xi)\in C^2(\mathbb{R})$ is a real function. Taking $L^2$ inner product of equations in \eqref{eqn:Vmodified_willmore_gradient_flow} with $\mu$, $\phi_t$ and $2u$ respectively, we can obtain
\begin{align}
    \frac{d}{dt}\mathcal{E}(u,\phi)=-\left\|\sqrt{M(\phi)}\nabla\mu\right\|^2\leq 0.
\end{align}
Then, we discretize the above modified model\eqref{eqn:Vmodified_willmore_gradient_flow} using the variable-time-step WSBDF2 method.

\begin{itemize}
    \item {\bf{$\mathcal{V}_W$-method}.} Given $\phi^n$, $u^n$ and $\phi^{n-1}$, update $\phi^{n+1}$, $u^{n+1}$ by solving
\end{itemize}
\begin{subequations}\label{eqn:VST-IMEX-WBDF_willmore}
    \begin{align}
        &\Tilde{D}_{2}^{n+\theta}\phi^{n+\theta}=\nabla\cdot(M^{*,n+\theta}\nabla\mu^{n+\theta}),\label{eqn:IMEX-VST_WBDF_willmore_1}\\
        &\mu^{n+\theta}=\frac{u^{n+1}}{\sqrt{\widetilde{E}_2^{n}}}V(\xi^{n+1})\mathcal{Z}^{*,n+\theta}+\frac{S_1}{\epsilon^2}\phi^{n+\theta}-S_2\Delta\phi^{n+\theta}+S_3\Delta^2\phi^{n+\theta},\label{eqn:IMEX-VST_WBDF_willmore_2}\\
        &\frac{u^{n+1}-u^n}{\tau_{n+1}}=\frac{V(\xi^{n+1})}{2\sqrt{\widetilde{E}_2^{n}}}\int_{\Omega}\mathcal{Z}^{*,n+\theta}\frac{\phi^{n+1}-\phi^n}{\tau_{n+1}}d\Omega.\label{eqn:IMEX-VST_WBDF_willmore_3}
    \end{align}
\end{subequations}
The following two theorems demonstrate that the structure-preserving properties still hold for the Willmore regularization system. Additionally, since the proof processes are similar to Theorem \ref{th:V_Energy_linear} and Theorem \ref{th:V_Mass_linear}, the proofs are omitted here.
\begin{theorem}
    For $\theta\in[\frac{1}{2},1]$ and $0\leq\gamma_{n+1}\leq\gamma_{*}$, if $M(\phi)$ in \eqref{eqn:willmore_gradient_flow} is a positive constant, the $\mathcal{V}_W$-method is unconditionally energy stable in the sense that
    \begin{align}
        \tilde{\mathcal{E}}_{W}^{n+1}-\tilde{\mathcal{E}}_{W}^{n}\leq 0,
    \end{align}
    where the modified energy is defined by
    \begin{align}
      \begin{aligned}\tilde{\mathcal{E}}_{W}^{n+1}=&\frac{(2\theta-1)\gamma_{n+2}^{3/2}}{2(1+\gamma_{n+2})}\frac{\|\nabla^{-1}(\phi^{n+1}-\phi^{n})\|^2}{M\tau_{n+1}}+\frac{\lambda_1}{2\epsilon^2}\|\phi^{n+1}\|^2\\
      &+\frac{\lambda_2}{2}\|\nabla\phi^{n+1}\|^2+\frac{\lambda_3}{2}\|\Delta\phi^{n+1}\|^2+|u^{n+1}|^2.
      \end{aligned}
    \end{align}
\end{theorem}
\begin{theorem}
         The solution of the $\mathcal{V}_L$-method \eqref{eqn:VST-IMEX-WBDF_willmore} satisfies the mass conservation.
     \end{theorem}

\begin{rem}
    Taking $\gamma_{n+1}=1$ in \eqref{eqn:VTS-WBDF2-Linear} and \eqref{eqn:VST-IMEX-WBDF_willmore}, the $\mathcal{V}_L$-method and $\mathcal{V}_M$-method can degenerate into their corresponding $\mathcal{U}_L$-method and $\mathcal{U}_W$-method on the uniform temporal mesh.
     However, the technique for constructing discrete energy in this section requires that $M(\phi)$ be a positive constant or a time-dependent positive function. However, the $\mathcal{U}_L$-method and $\mathcal{U}_M$-method we propose can maintain unconditional energy stability for any $M(\phi)$.
\end{rem}
\begin{rem}
     Introducing the second order stabilization term, $\phi^{n+1}-(1+\gamma_{n+1})\phi^n+\gamma_{n+1}\phi^{n-1}$,in the $\mathcal{V}_L$-method and $\mathcal{V}_W$-method, similar to the technique used in the $\mathcal{U}_L$-method and $\mathcal{U}_W$-method, results in our inability to construct discrete energy formulations that preserve unconditional energy stability. Additionally, If we choose to add second-order stabilization terms, such as $\tau(\phi^{n+1}-\phi^n)$ and $\tau^2\phi^{n+1}$, which are easy to construct discrete energy formulations that maintain unconditional energy stability, results in poor stability in practical computation. Therefore, we employ a different technique for adding stabilization terms.
\end{rem}

\section{Numerical simulations}\label{section_5}

In this section, we provide several numerical examples to verify the accuracy, mass conservation and energy dissipation of the proposed schemes.
In all the tests, we discretize space using the Fourier spectral method \cite{shen2011spectral,ainsworth2017analysis,chen2018enriched}.

In the absence of explicit specifications, the parameters for the numerical experiments are set as follows
\begin{align}\label{eqn:parameters}
    M(\phi)=1, \epsilon=0.2,\beta=6e-4,S_1=4,S_2=4,\lambda_1=0,\lambda_2=4.
\end{align}
\subsection{Numerical simulations in 1D}
In this subsection, we perform numerical simulations, include accuracy and structure-preservation of the solutions, for the linear regularization model with $\mathcal{U}_L$-method \eqref{eqn:Uniform-WBDF2-Linear} and $\mathcal{V}_L$-method \eqref{eqn:VTS-WBDF2-Linear} in 1D. Here, the computational domain is defined as $\Omega=[0,2\pi]$, with the mesh size $N_x=128$.

\textbf{Example 1.}
In this example, we aim to test the convergence rates of $\mathcal{U}_L$-method and $\mathcal{V}_L$-method in various parameter settings. To this end, we add a source term in the original model, determined by the exact solution
\begin{align}\label{eqn:initial-condition-1}
    \phi(x,t)=(t+1)^3sin(x).
\end{align}
Firstly, we plot the $L^2$-norm errors of the $\mathcal{U}_L$-method in Figure \ref{figure:error_uniform}.
We can observe from
Figure \ref{figure:error_uniform} that the error curves of the schemes under $S_1=0,S_2=0$ and $S_1=4,S_2=0$ are incomplete, which means that the corresponding schemes are unstable and their solutions may exhibit explosive growth during computation. In addition, the solutions under $S_1=0,S_2=4$ show better convergence than those under $S_1=4,S_2=4$ with small time steps due to the introduction of additional error terms when $S_1=4$. However, the solutions under $S_1=4,S_2=4$ consistently maintain good convergence, even with larger time steps, while those under $S_1=0,S_2=4$ do not.
From Figure \ref{figure:error_uniform}, we can also see that as the anisotropy intensity increases, the effect becomes more pronounced.
The $\mathcal{V}_L$-method on temporal mesh shows similar numerical results as the $\mathcal{U}_L$-method, see Figure \ref{figure:error_nonuniform_stronganiso}.  From the tests in this example, we conclude that the stabilization terms dramatically influence the convergence of the numerical schemes.

\par
\begin{figure}[!htp]
\centering
\includegraphics[width=0.32\textwidth]{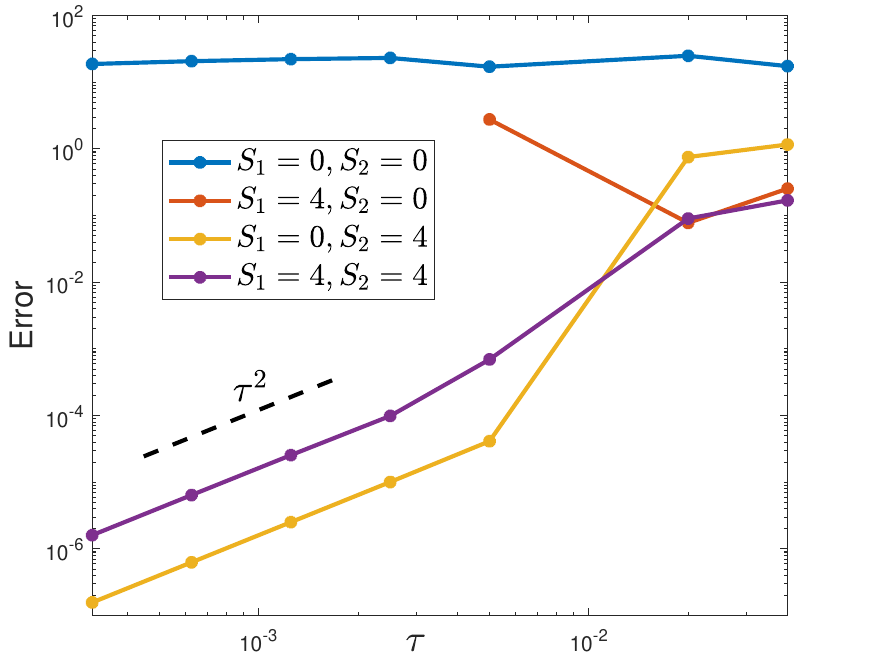}
\includegraphics[width=0.32\textwidth]{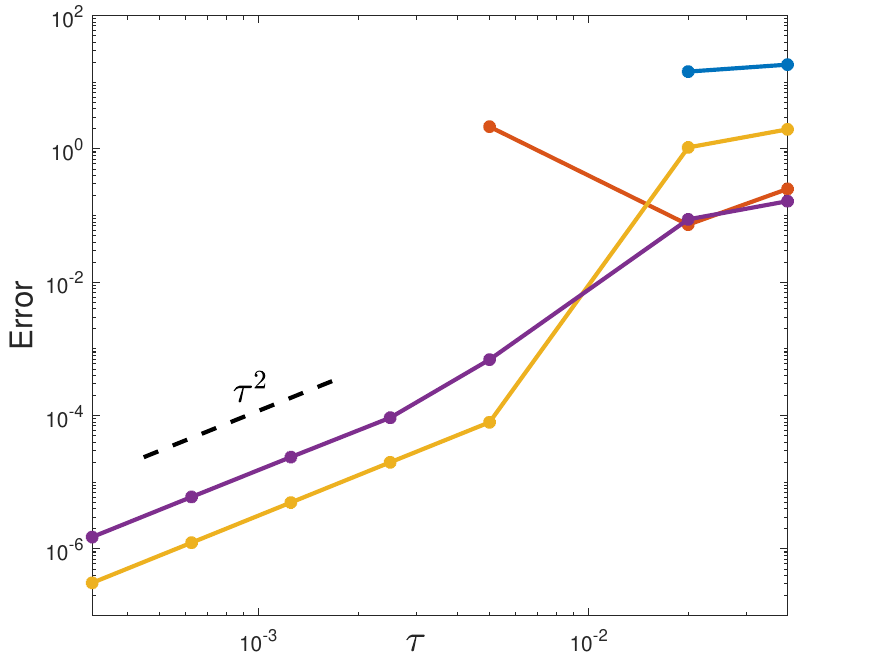}
\includegraphics[width=0.32\textwidth]{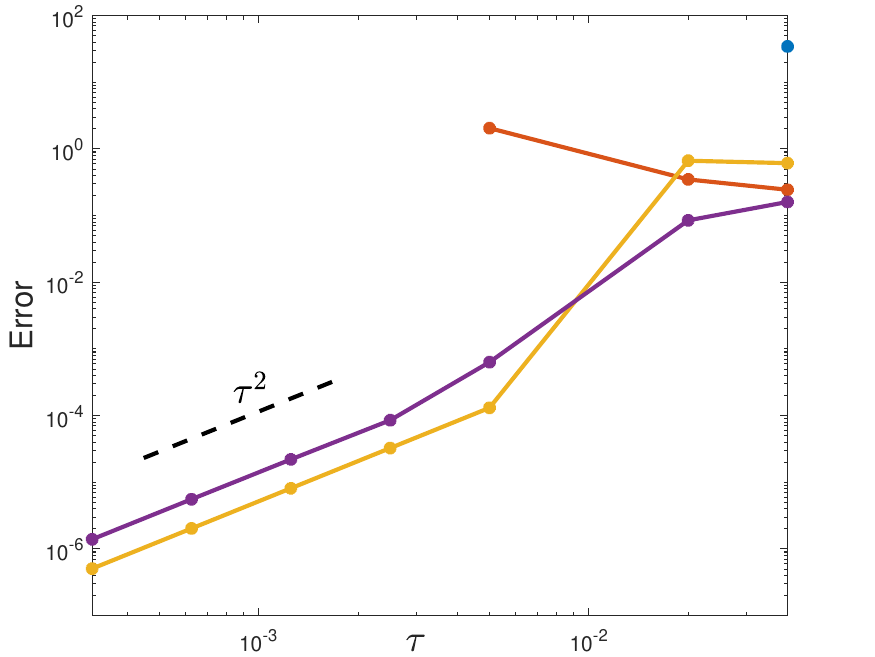}
\par
\centering
\includegraphics[width=0.32\textwidth]{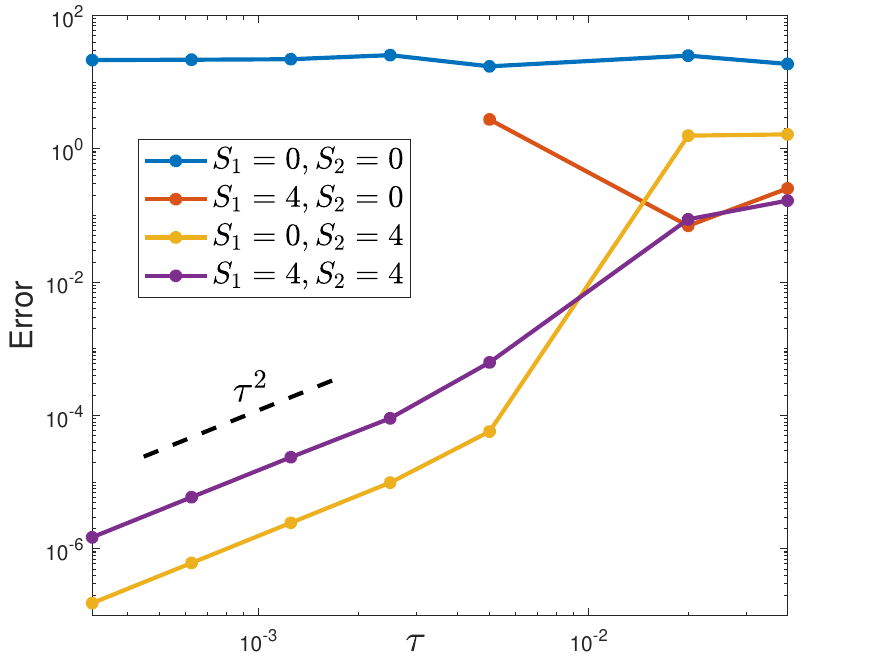}
\includegraphics[width=0.32\textwidth]{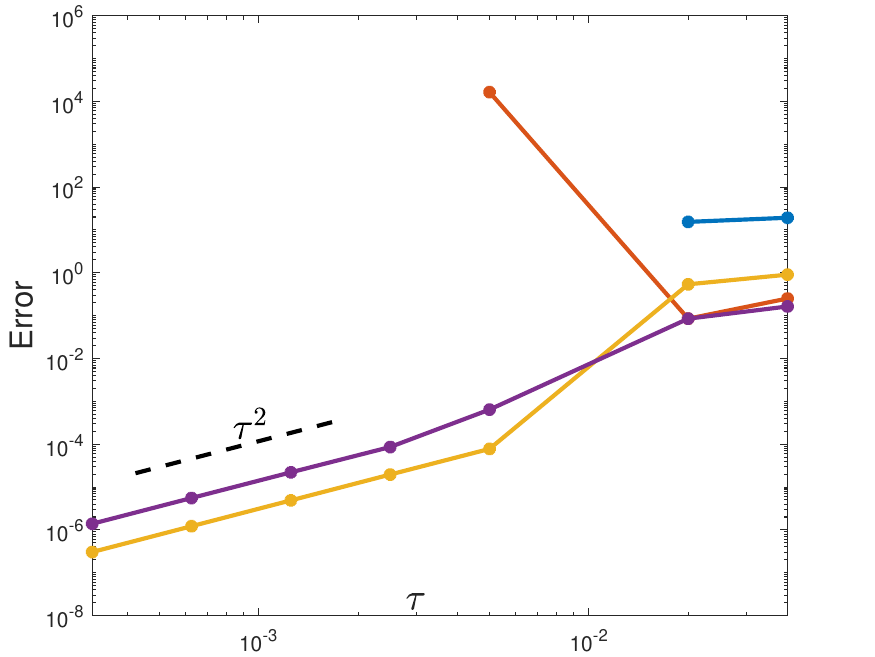}
\includegraphics[width=0.32\textwidth]{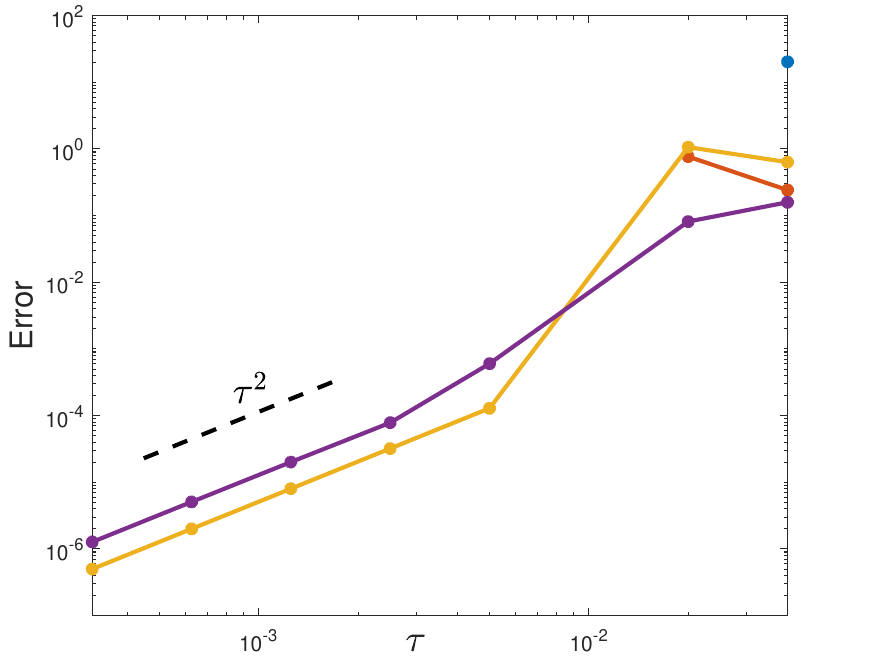}
\par
\centering
\includegraphics[width=0.32\textwidth]{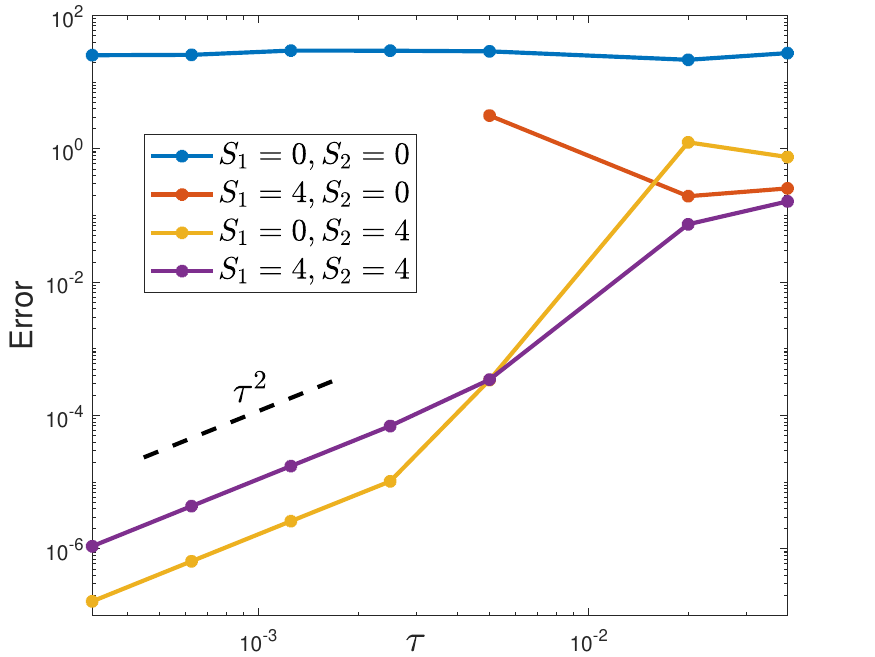}
\includegraphics[width=0.32\textwidth]{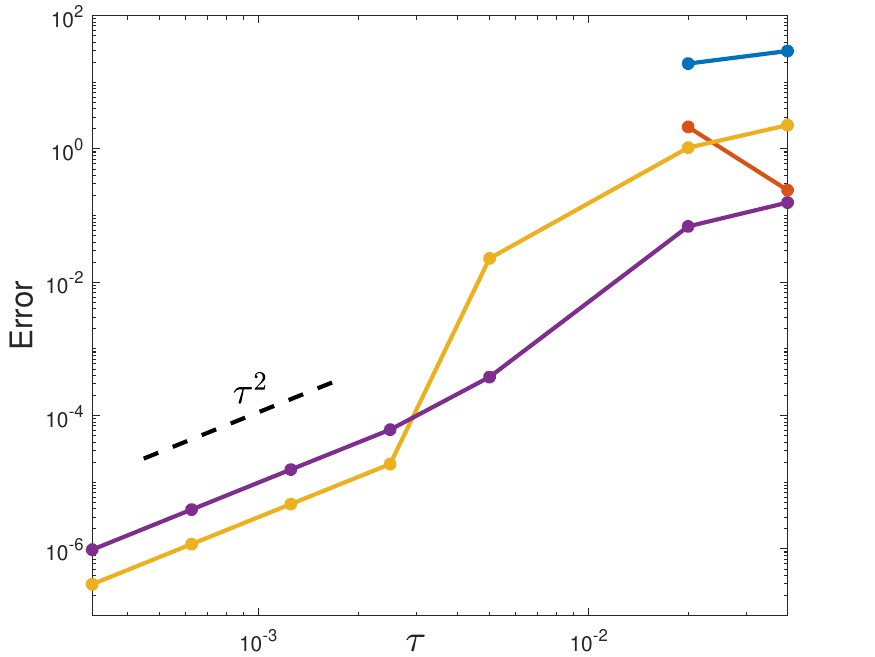}
\includegraphics[width=0.32\textwidth]{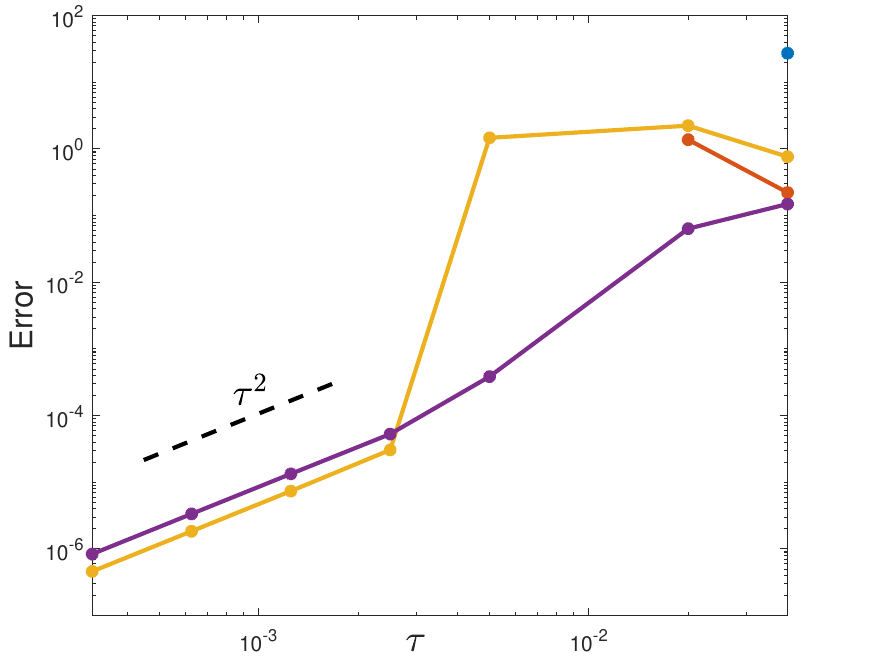}
\caption{$L^2$-norm errors for the uniform-time-step method with different $\theta$:
$\theta=0.5$ (first column), $\theta=0.75$ (second column), $\theta=1$ (last column), and with different $\alpha$:
$\alpha = 0$ (first row), $\alpha=0.05$ (second row) and $\alpha = 0.3$ (last row). The other parameters are selected by \eqref{eqn:parameters}.}
\label{figure:error_uniform}
\end{figure}

     \begin{figure}[!htp]
\centering
\includegraphics[width=0.32\textwidth]{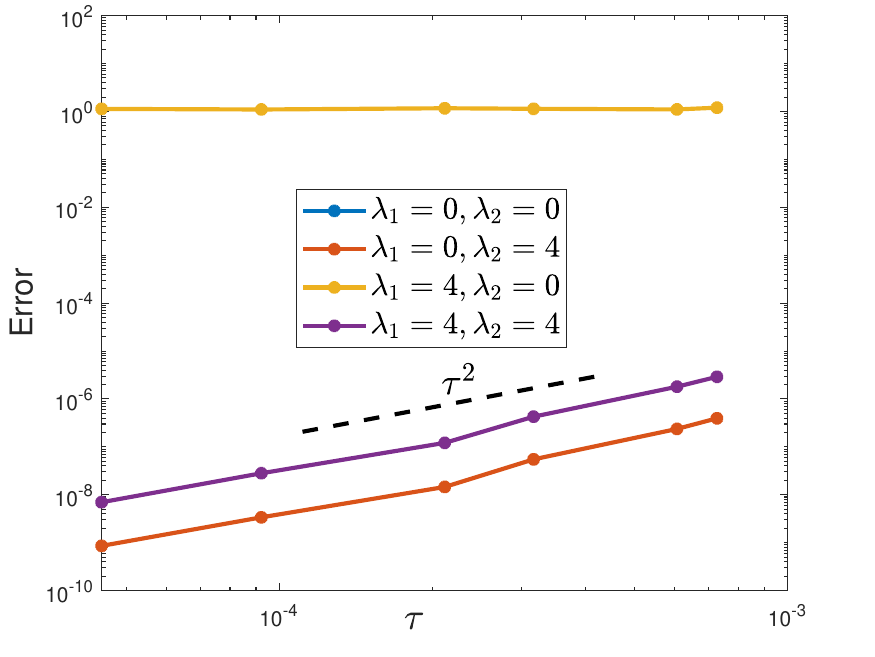}
\includegraphics[width=0.32\textwidth]{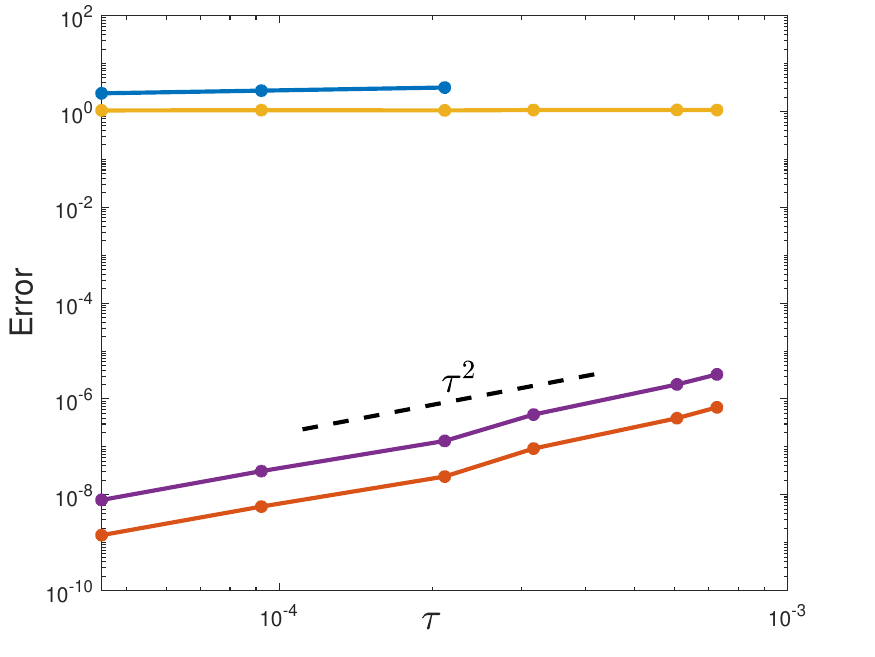}
\includegraphics[width=0.32\textwidth]{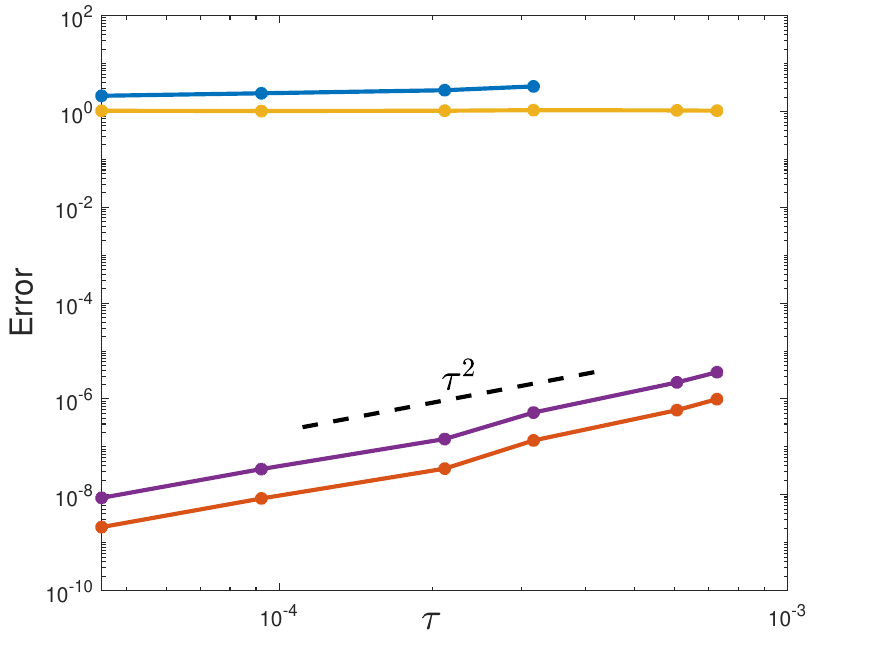}
\par
\centering
\includegraphics[width=0.32\textwidth]{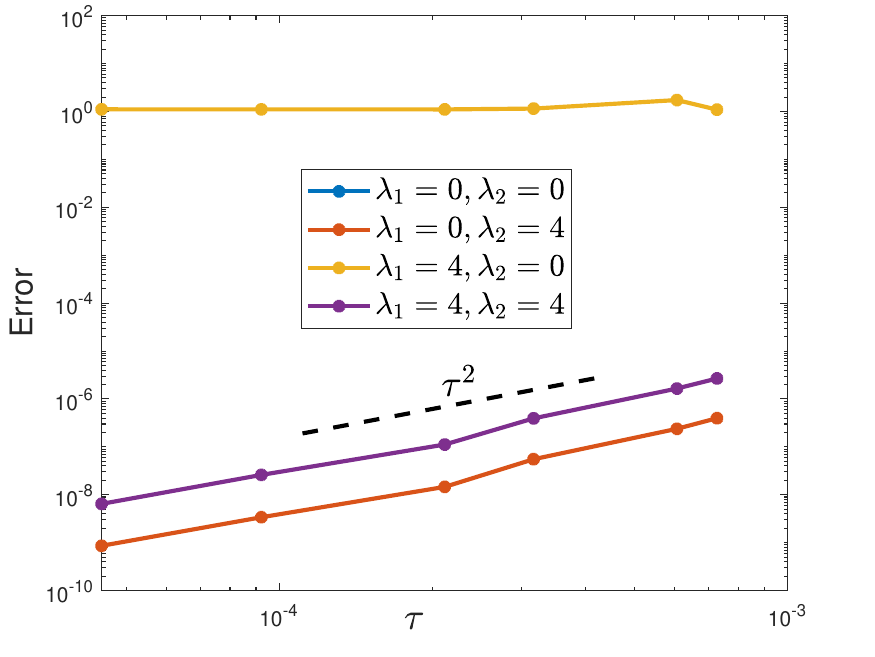}
\includegraphics[width=0.32\textwidth]{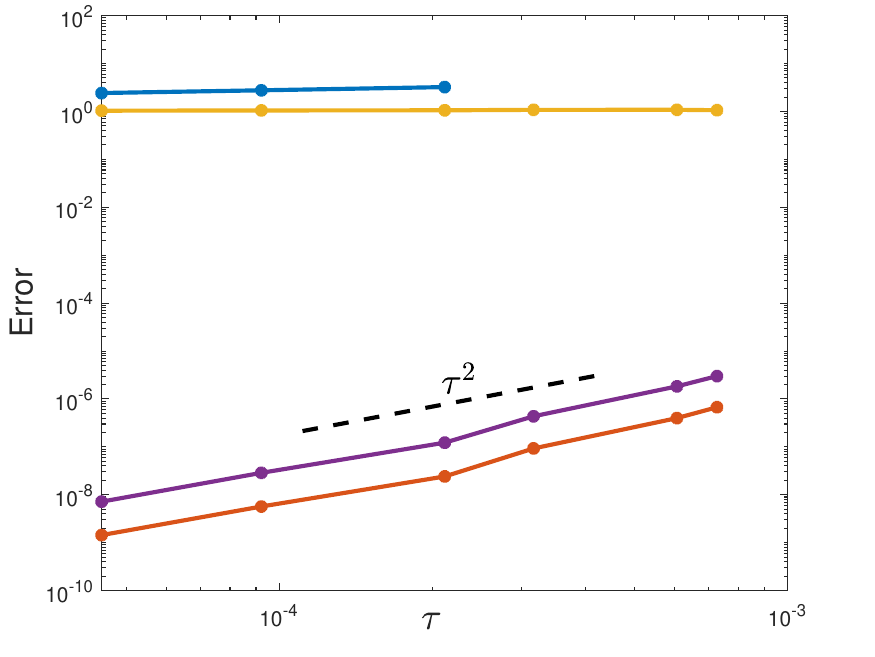}
\includegraphics[width=0.32\textwidth]{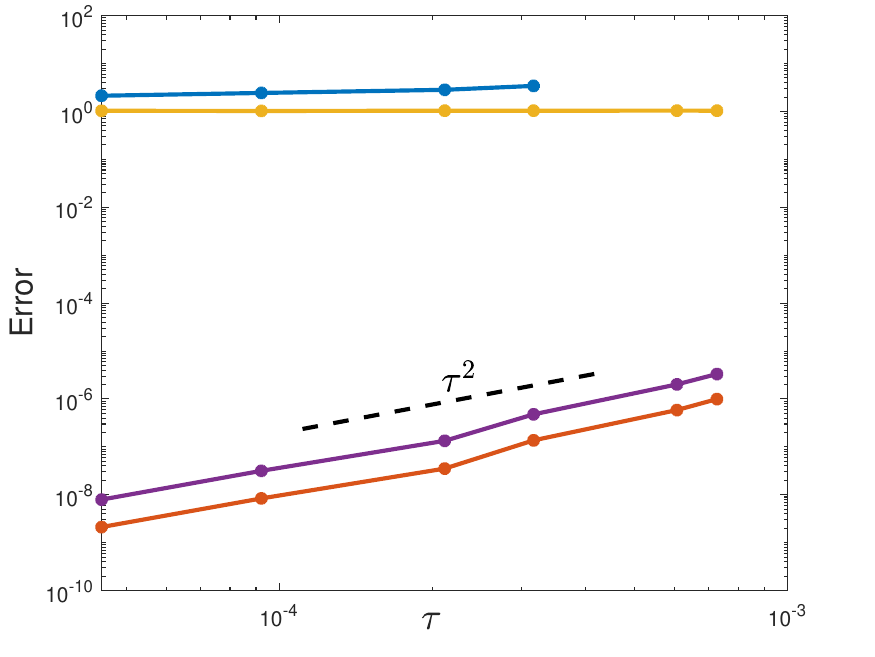}
\par
\centering
\includegraphics[width=0.32\textwidth]{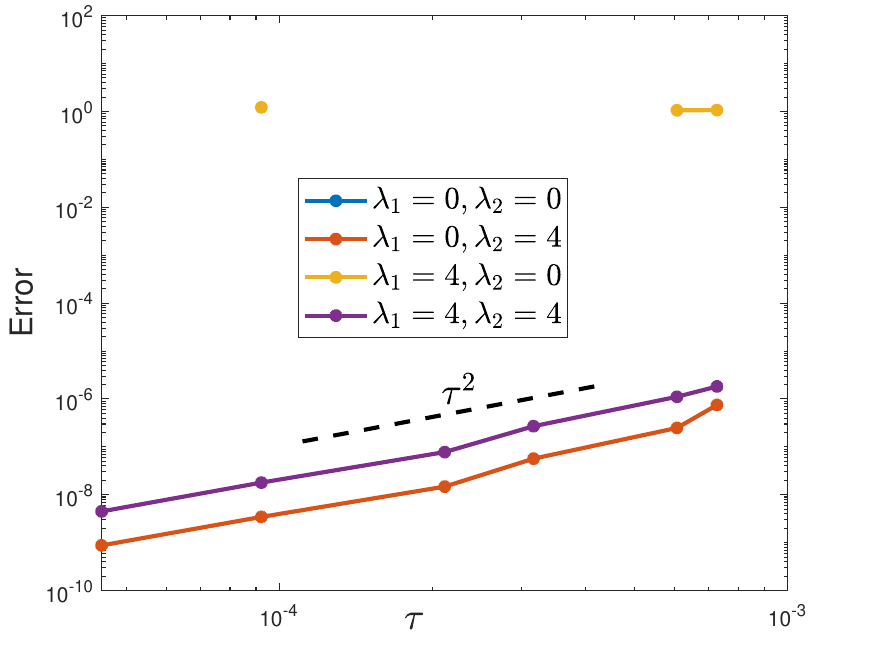}
\includegraphics[width=0.32\textwidth]{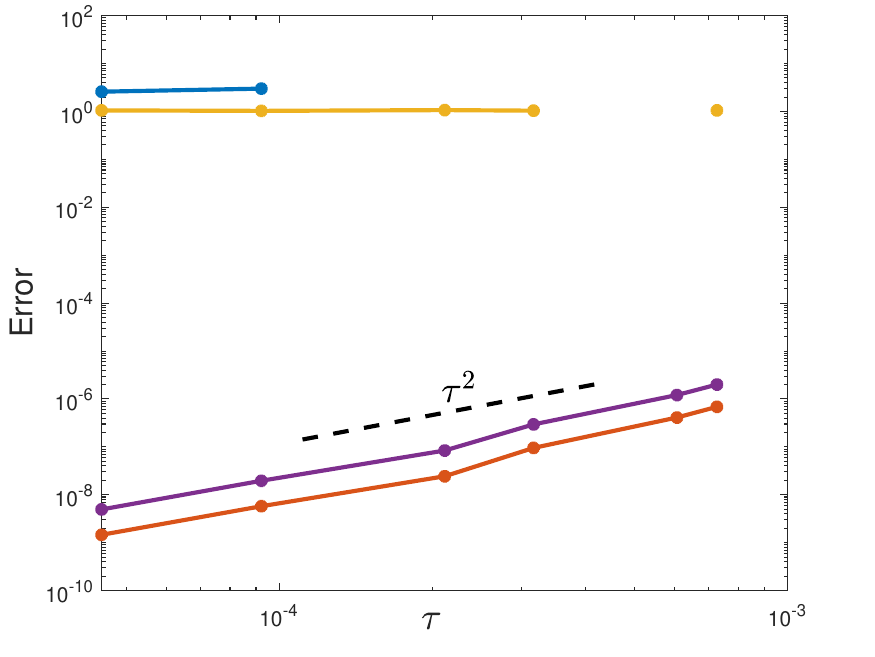}
\includegraphics[width=0.32\textwidth]{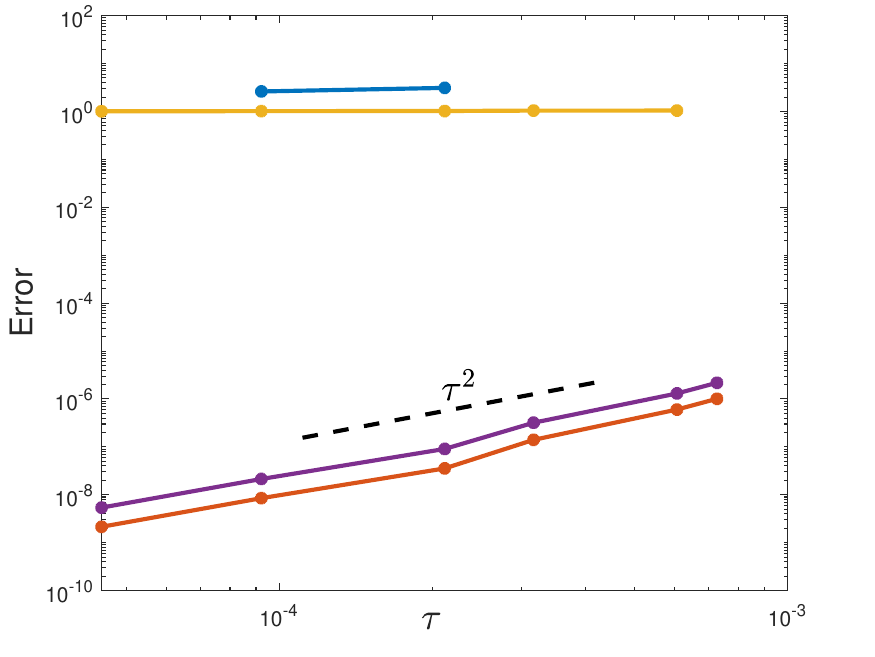}
\caption{$L^2$-norm errors for the variable-time-step method with different $\theta$:
$\theta=0.5$ (first column), $\theta=0.75$ (second column), $\theta=1$ (last column), and with different $\alpha$:
$\alpha = 0$ (first row), $\alpha=0.05$ (second row) and $\alpha = 0.3$ (last row). The other parameters are selected by \eqref{eqn:parameters}.}
\label{figure:error_nonuniform_stronganiso}
\end{figure}


\begin{figure}[!htp]
\centering
\includegraphics[width=0.32\textwidth]{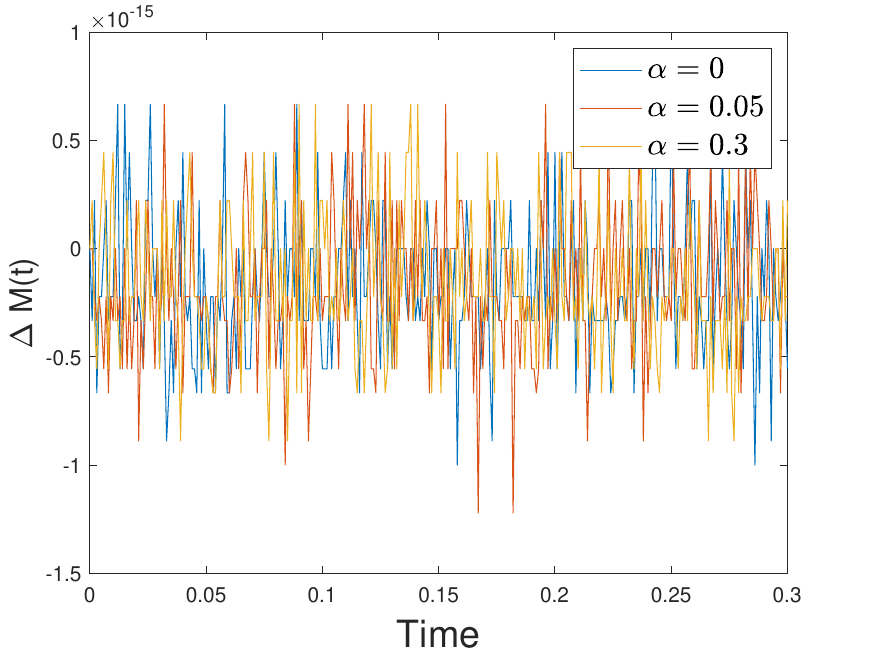}
\includegraphics[width=0.32\textwidth]{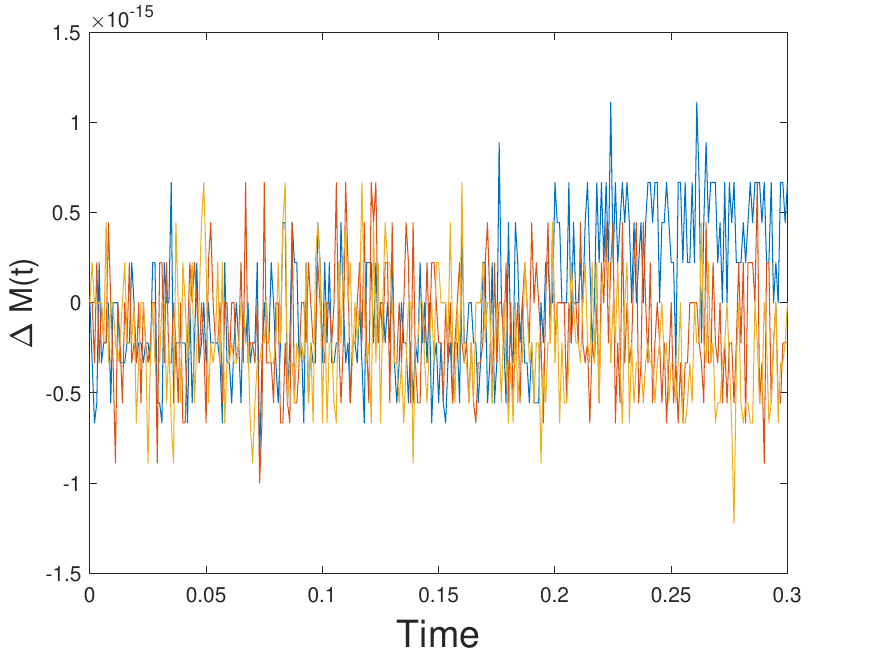}
\includegraphics[width=0.32\textwidth]{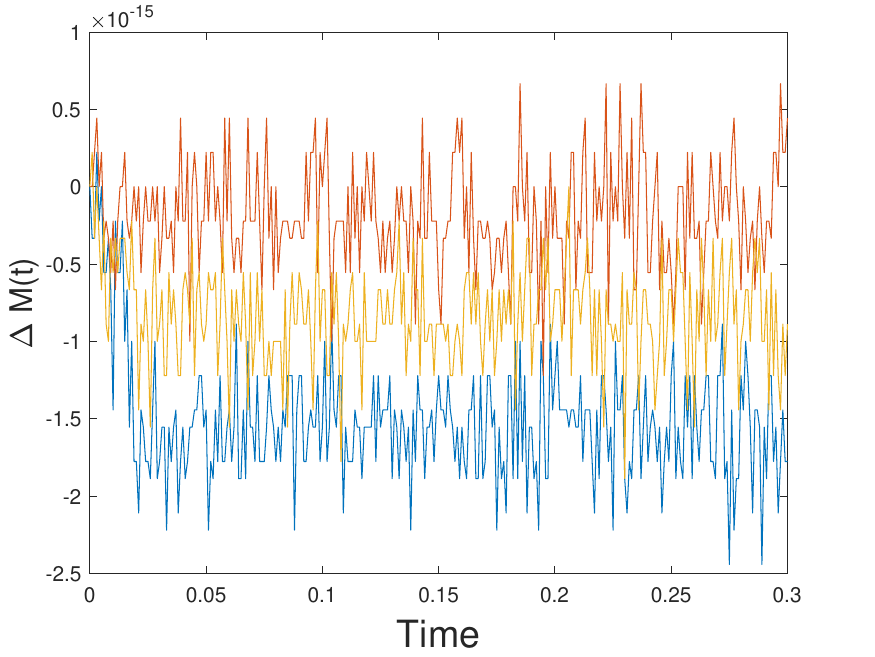}
\textit{(a) The relative errors of mass.}
\vspace{\baselineskip} 
\par
\includegraphics[width=0.32\textwidth]{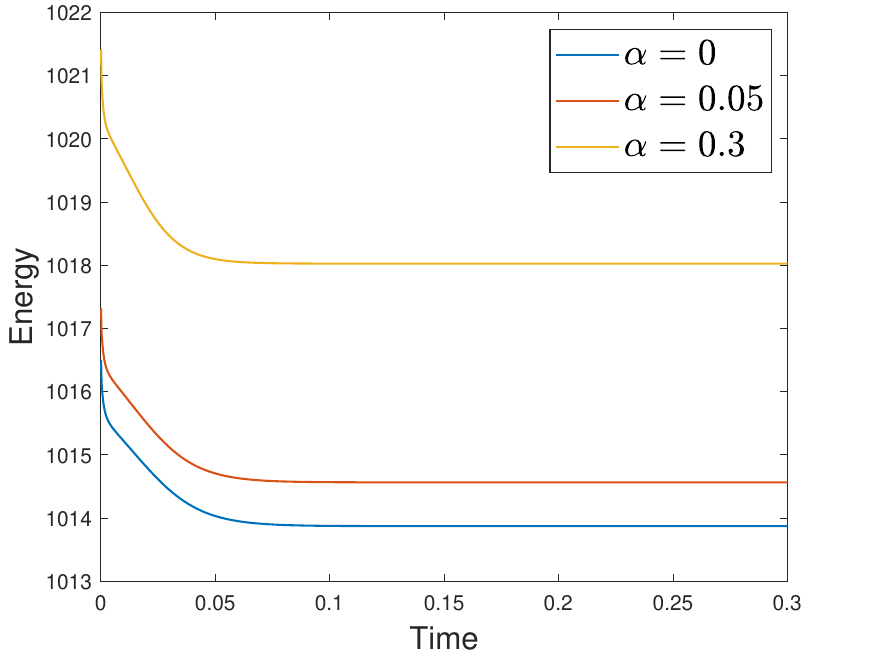}
\includegraphics[width=0.32\textwidth]{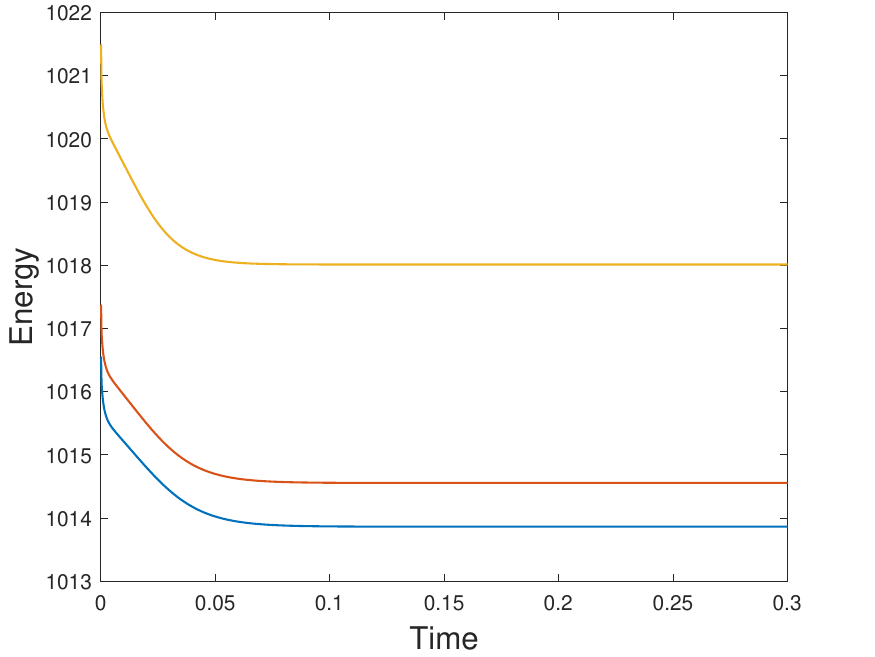}
\includegraphics[width=0.32\textwidth]{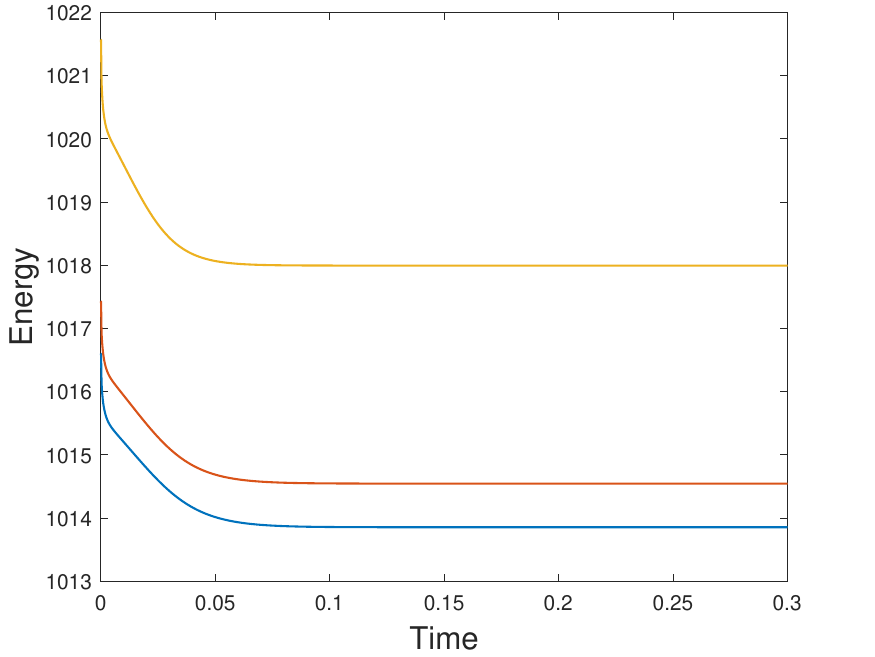}
\textit{(b) The energy evolutions.}
\caption{The relative error of mass and the energy evolutions for $\mathcal{U}_L$-method with different $\theta$:
$\theta=0.5$ (first column), $\theta=0.75$ (second column), $\theta=1$ (last column). The initial condition is chosen as \eqref{eqn:initial-condition_2_1} and the other parameters are selected by \eqref{eqn:parameters}. (a) The relative error of mass with $\tau=1e-3$. (b) The modified energy \eqref{eqn:UE_linear} with $\tau=1e-3$.}
\label{fig:ME_uniform_sinx1D}
\end{figure}

\begin{figure}[!htp]
\centering
\includegraphics[width=0.32\textwidth]{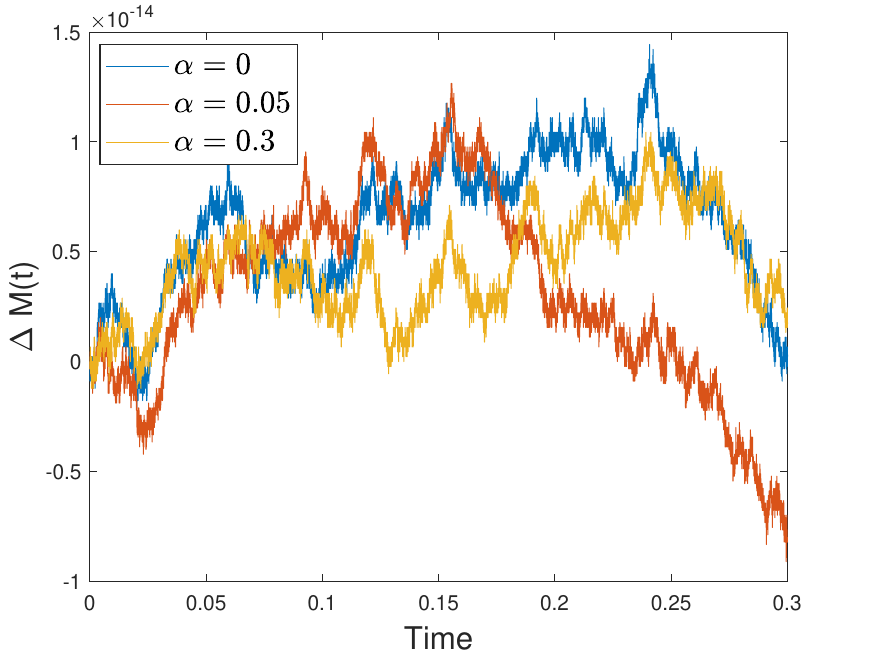}
\includegraphics[width=0.32\textwidth]{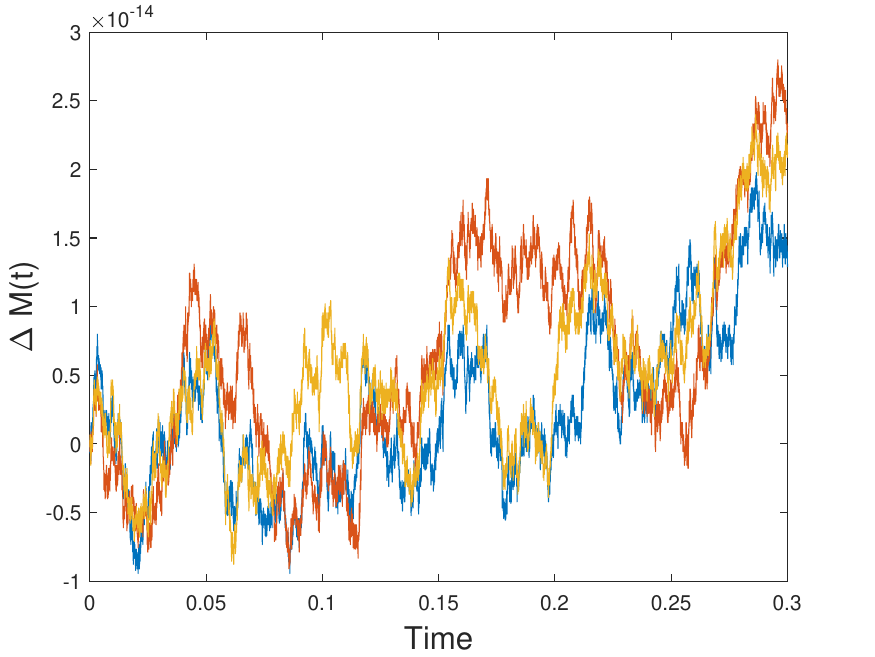}
\includegraphics[width=0.32\textwidth]{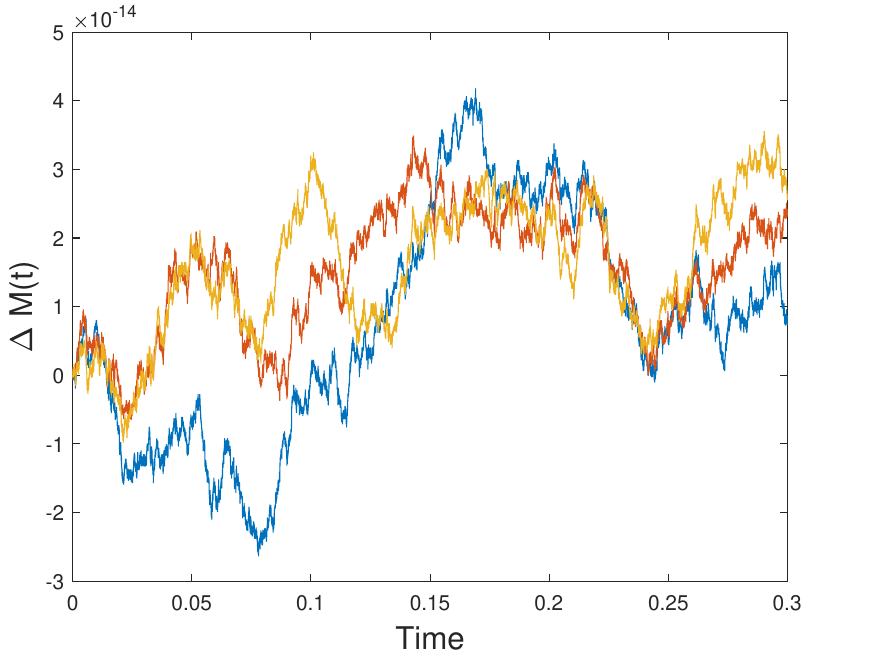}
\textit{(a) The relative errors of mass.}
\vspace{\baselineskip} 
\par
\includegraphics[width=0.32\textwidth]{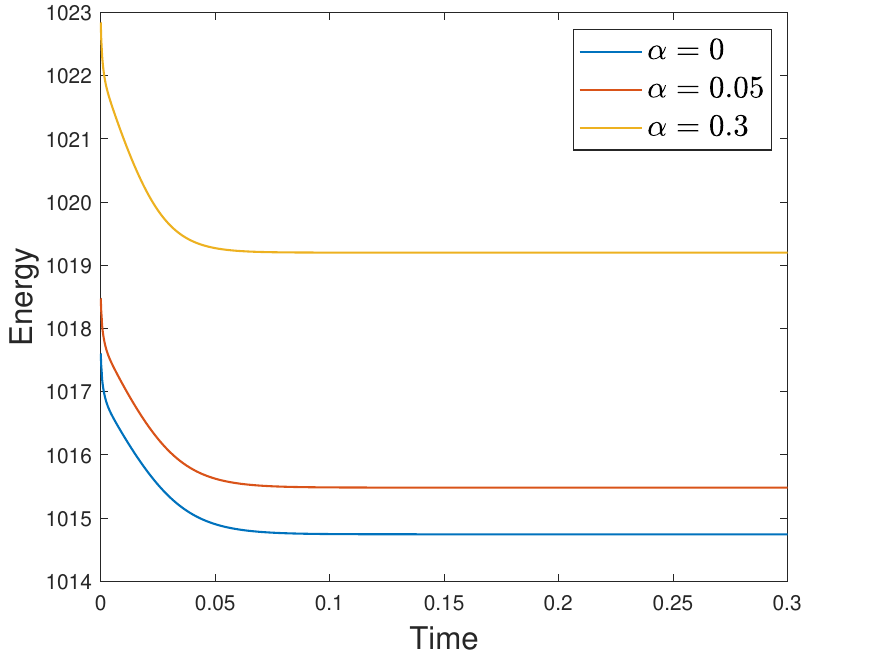}
\includegraphics[width=0.32\textwidth]{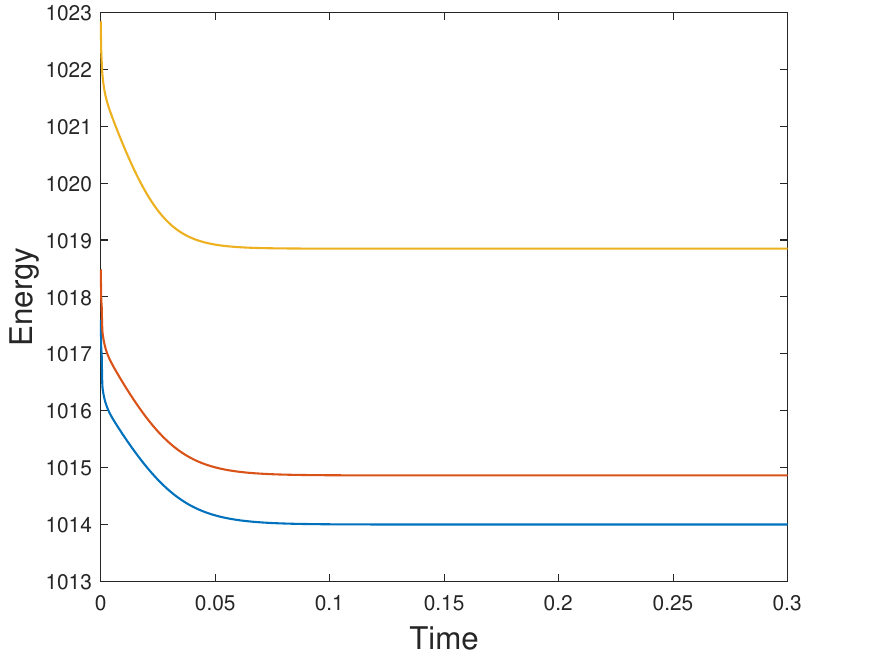}
\includegraphics[width=0.32\textwidth]{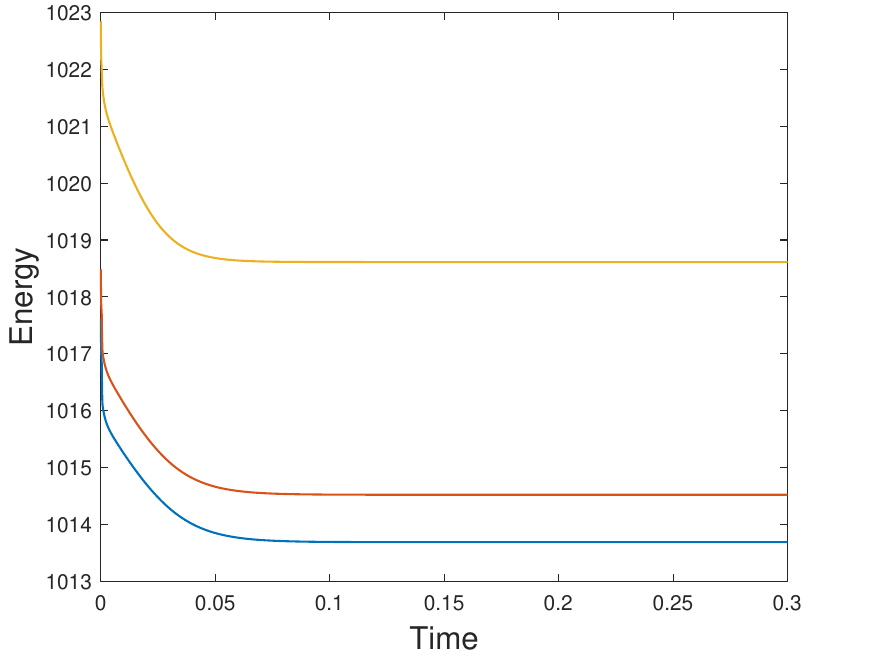}
\textit{(b) The energy evolutions.}
\caption{The relative error of mass and the energy evolutions for $\mathcal{V}_L$-method with different $\theta$:
$\theta=0.5$ (first column), $\theta=0.75$ (second column), $\theta=1$ (last column). The initial condition is chosen as \eqref{eqn:initial-condition_2_1} and the other parameters are selected by \eqref{eqn:parameters}. (a) The relative error of mass with $\tau_{max}=1.0165e-4$. (b) The modified energy \eqref{eqn:VE_linear} with $\tau_{max}=1.0165e-4$.}
\label{fig:ME_VTS_sinx1D}
\end{figure}
\textbf{Example 2.}
In this example, we present the evolutions of the relative mass error defined by $\Delta M(t)=\frac{M(t)-M(0)}{M(0)}$ and the evolutions of the energy function for different values of $\alpha$ and $\theta$.
We consider the following two types of initial conditions:
\begin{align}
  & \qquad   \phi(x,0)=|sin(x)|\,,\label{eqn:initial-condition_2_1}\\
  & \qquad   \phi(x,0)=-0.3 + 0.001rand(x)\,.\label{eqn:initial-condition_2_2}
\end{align}
For various values of the anisotropy intensity parameter $\alpha$ and the weighted parameter $\theta$,
the relative errors of mass and the energy for the  $\mathcal{U}_L$-method and $\mathcal{V}_L$-method are plotted in
Figures \ref{fig:ME_uniform_sinx1D}, \ref{fig:ME_VTS_sinx1D}, \ref{fig:ME_uniform_rand1D} and \ref{fig:ME_VTS_rand1D}.
We can conclude that both $\mathcal{U}_L$-method and $\mathcal{V}_L$-method can preserve mass conservation and energy dissipation very well, which are consistent with our theoretical analysis.

\begin{figure}[!htp]
\centering
\includegraphics[width=0.32\textwidth]{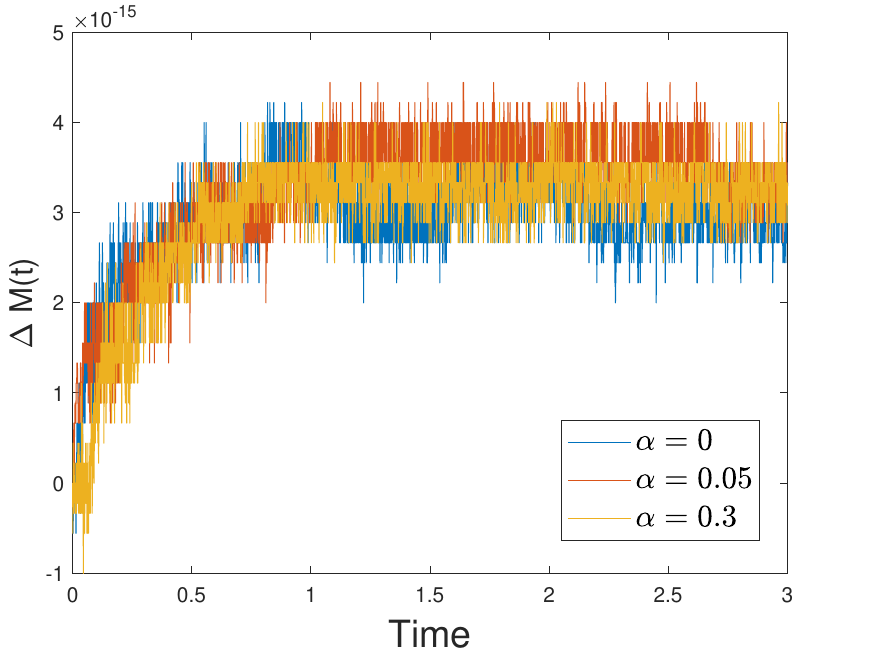}
\includegraphics[width=0.32\textwidth]{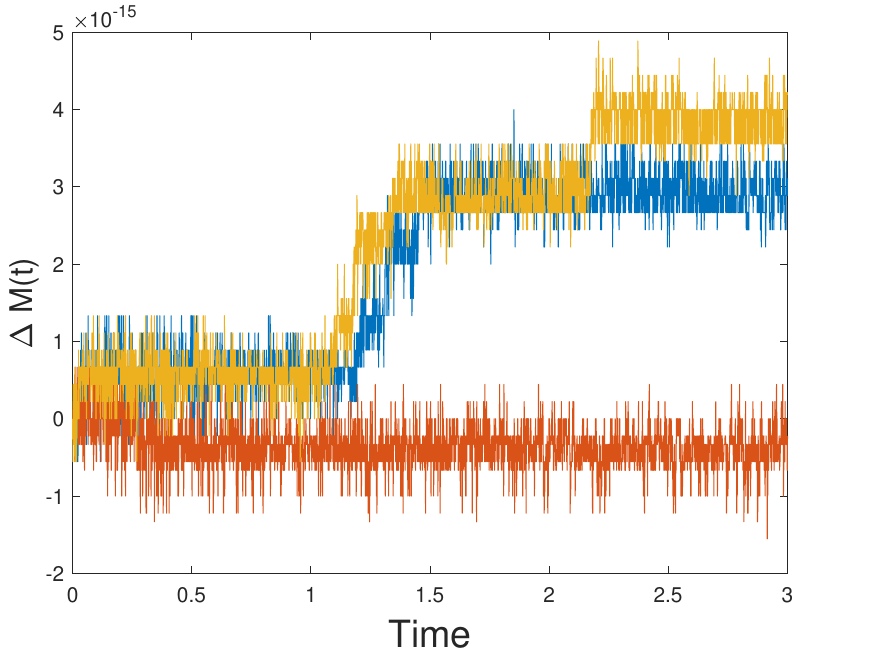}
\includegraphics[width=0.32\textwidth]{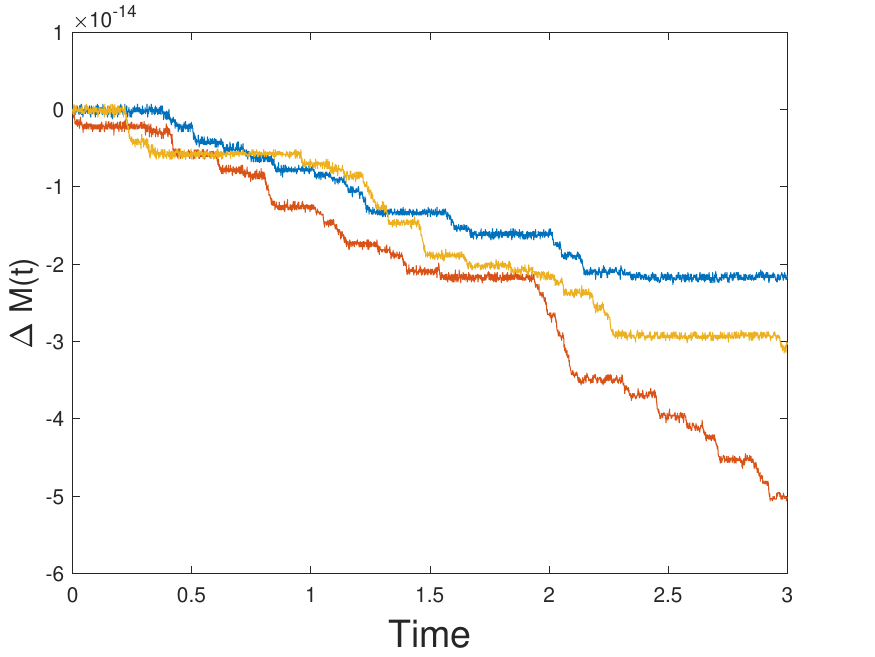}
\textit{(a) The relative error of mass.}
\vspace{\baselineskip} 
\par
\includegraphics[width=0.32\textwidth]{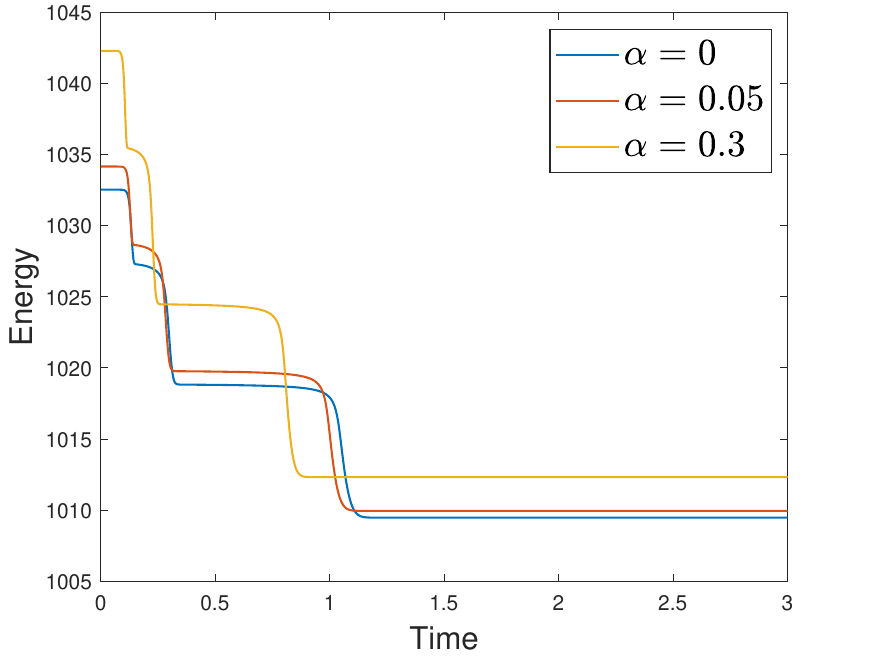}
\includegraphics[width=0.32\textwidth]{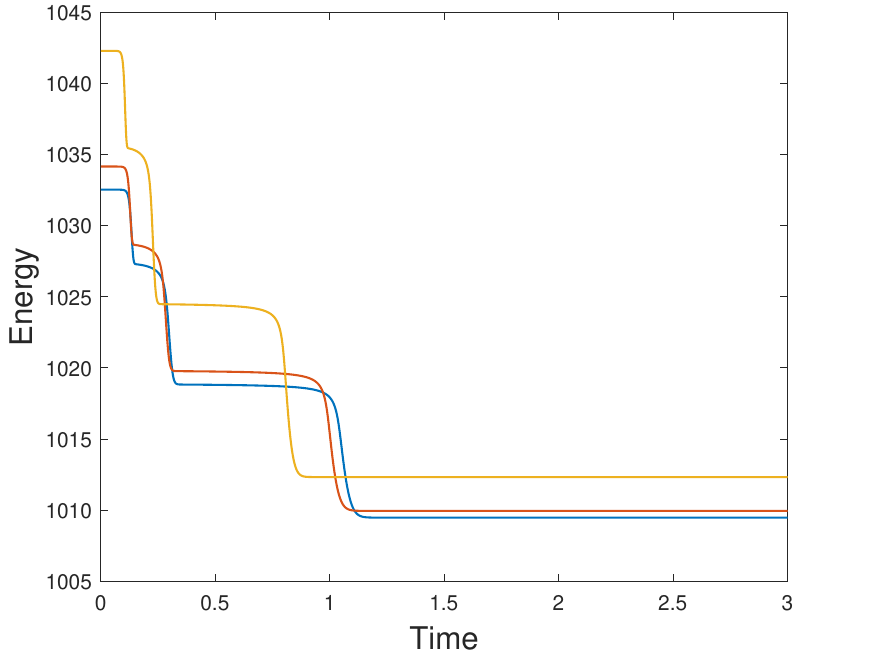}
\includegraphics[width=0.32\textwidth]{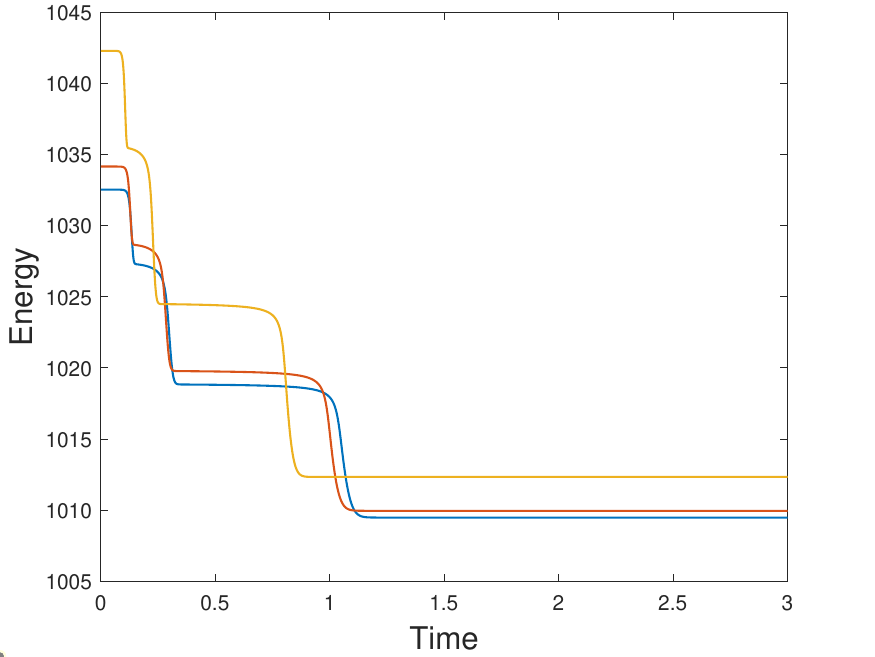}
\textit{(b) The energy evolutions.}
\caption{The relative error of mass and the energy evolutions for $\mathcal{U}_L$-method with different $\theta$:
$\theta=0.5$ (first column), $\theta=0.75$ (second column), $\theta=1$ (last column). The random initial condition \eqref{eqn:initial-condition_2_2} is chosen and the other parameters are selected by \eqref{eqn:parameters}. (a) The relative error of mass with $\tau=1e-3$. (b) The modified energy \eqref{eqn:VE_linear} with $\tau=1e-3$.}
\label{fig:ME_uniform_rand1D}
\end{figure}

\begin{figure}[!htp]
\centering
\includegraphics[width=0.32\textwidth]{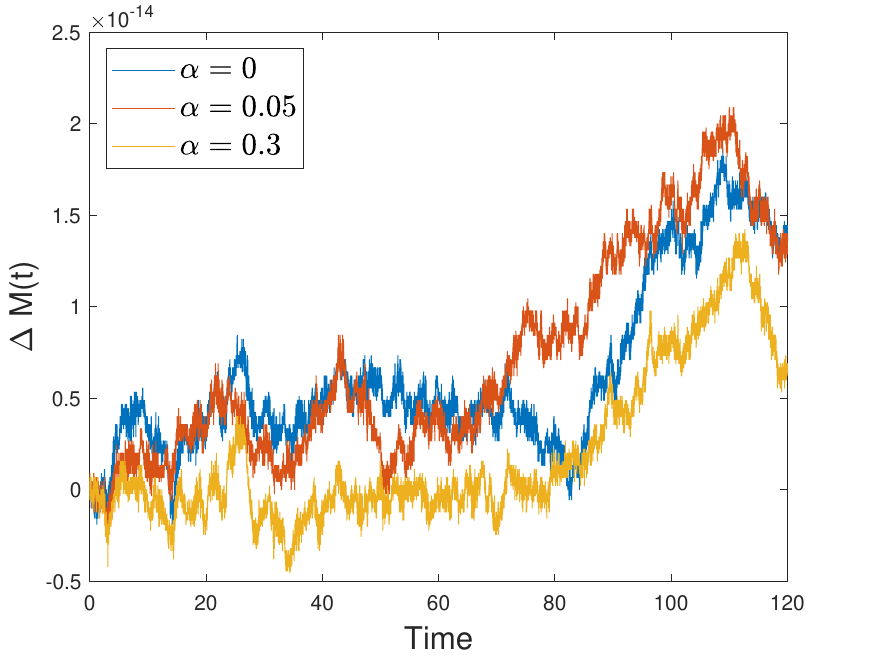}
\includegraphics[width=0.32\textwidth]{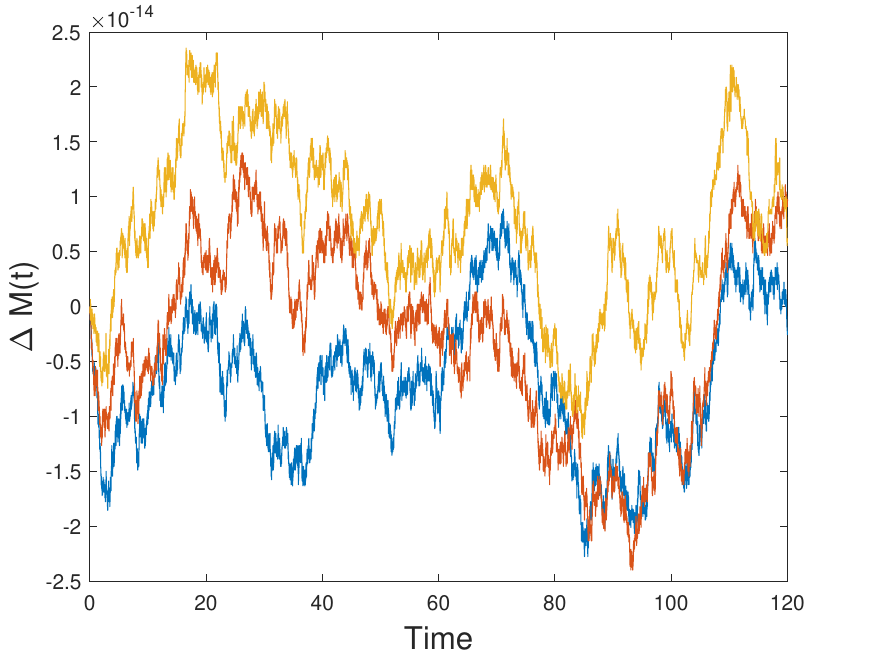}
\includegraphics[width=0.32\textwidth]{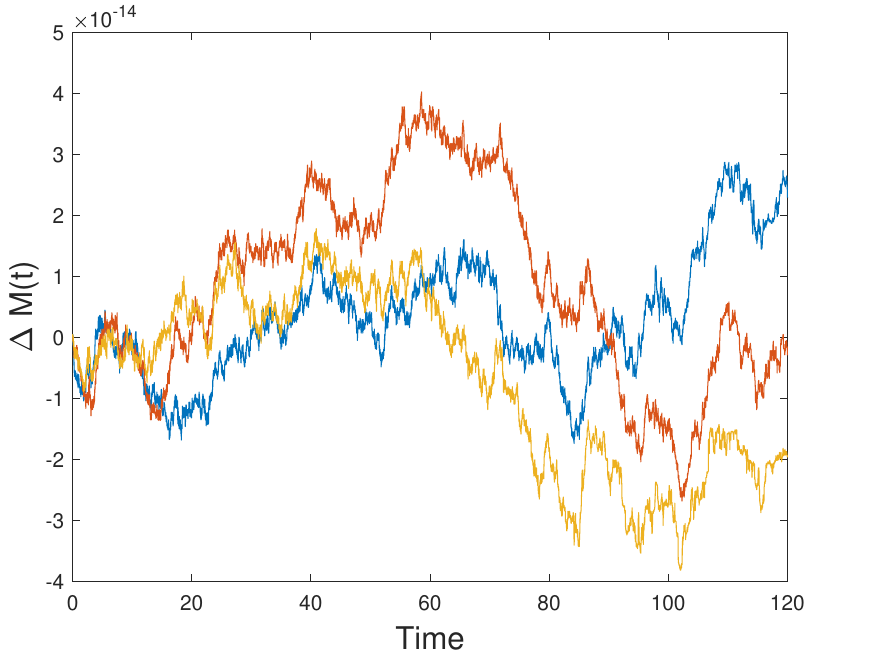}
\textit{(a) The relative error of mass.}
\vspace{\baselineskip} 
\par
\includegraphics[width=0.32\textwidth]{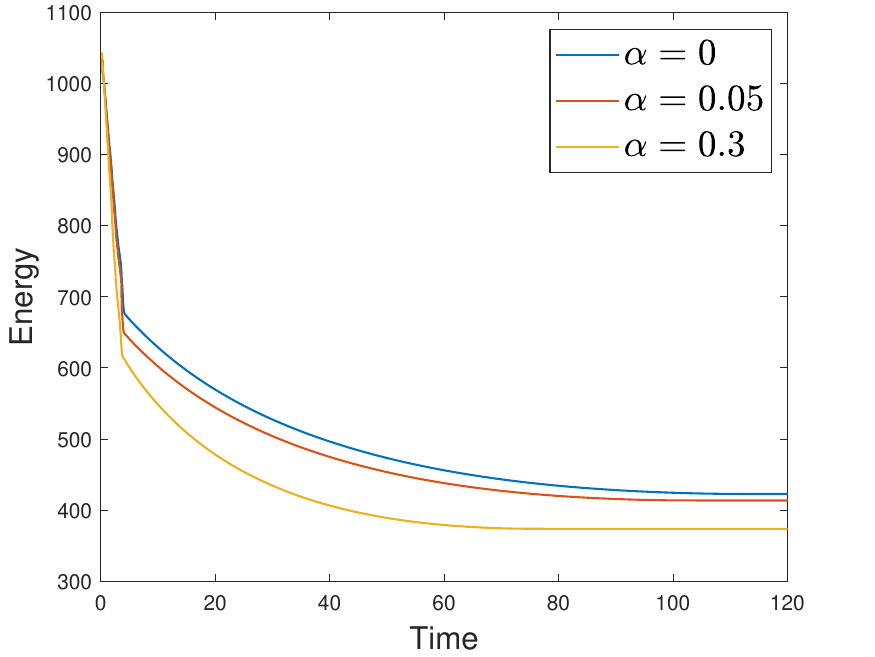}
\includegraphics[width=0.32\textwidth]{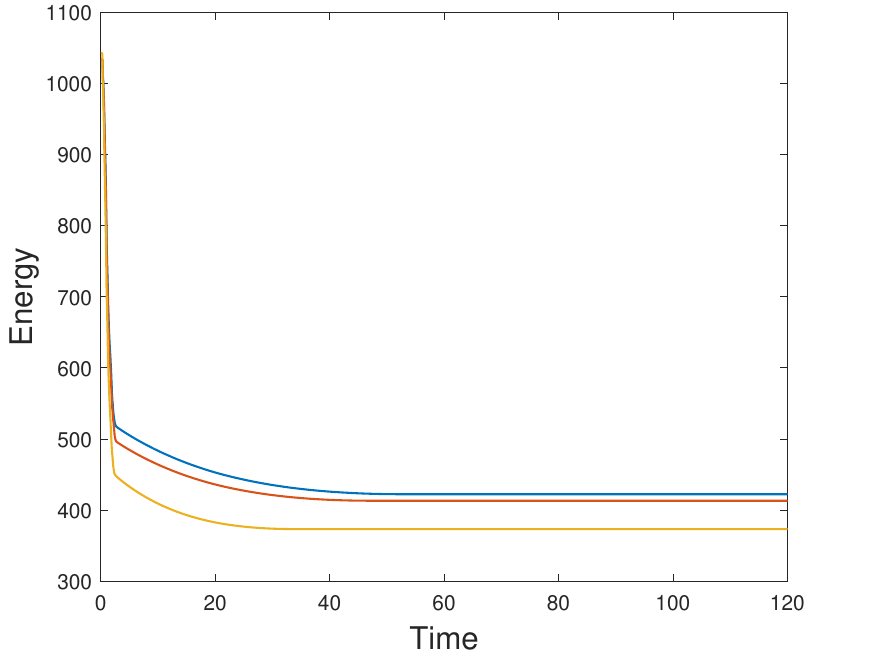}
\includegraphics[width=0.32\textwidth]{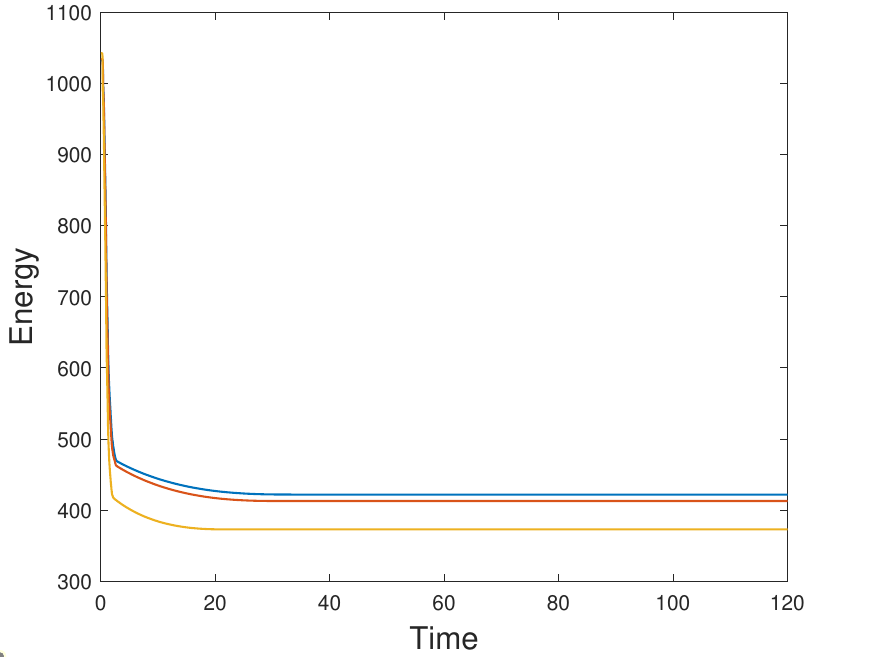}
\textit{(b) The energy evolutions.}
\caption{The relative error of mass and the energy evolutions for $\mathcal{V}_L$-method with different $\theta$:
$\theta=0.5$ (first column), $\theta=0.75$ (second column), $\theta=1$ (last column). The random initial condition \eqref{eqn:initial-condition_2_2} is chosen and the other parameters are selected by \eqref{eqn:parameters}. (a) The relative error of mass with $\tau_{max}=1.0165e-4$. (b) The modified energy \eqref{eqn:VE_linear} with $\tau_{max}=1.0165e-4$.}
\label{fig:ME_VTS_rand1D}
\end{figure}


\subsection{ Numerical simulations in 2D}

In this subsection, we perform numerical simulations, include structure-preservation of the solutions and temporal evolution, for the linear regularization model with $\mathcal{U}_L$-method \eqref{eqn:Uniform-WBDF2-Linear} and $\mathcal{V}_L$-method \eqref{eqn:VTS-WBDF2-Linear} in 2D. Here, the computational domain is defined as $\Omega=[0,2\pi]\times[0,2\pi]$ with the mesh size $N_x=128,Ny=128$.

\textbf{Example 3.} In this example, we numerical simulate the evolution of two circles, and the initial condition is given by
\begin{align}\label{eqn:initial_condition_3_1}
	\phi(x,y,0)=\sum_{i=1}^{2}-tanh(\frac{\sqrt{(x-x_i)^2+(y-y_i)^2}-r_i}{1.2\epsilon})+1,
\end{align}
where $(x_1,y_1,r_1)=(\pi-0.7,\pi-0.6,1.5)$ and $(x_1,y_1,r_1)=(\pi+1.65,\pi+1.6,0.7)$.

\begin{figure}[!htp]
\centering
\includegraphics[width=0.32\textwidth]{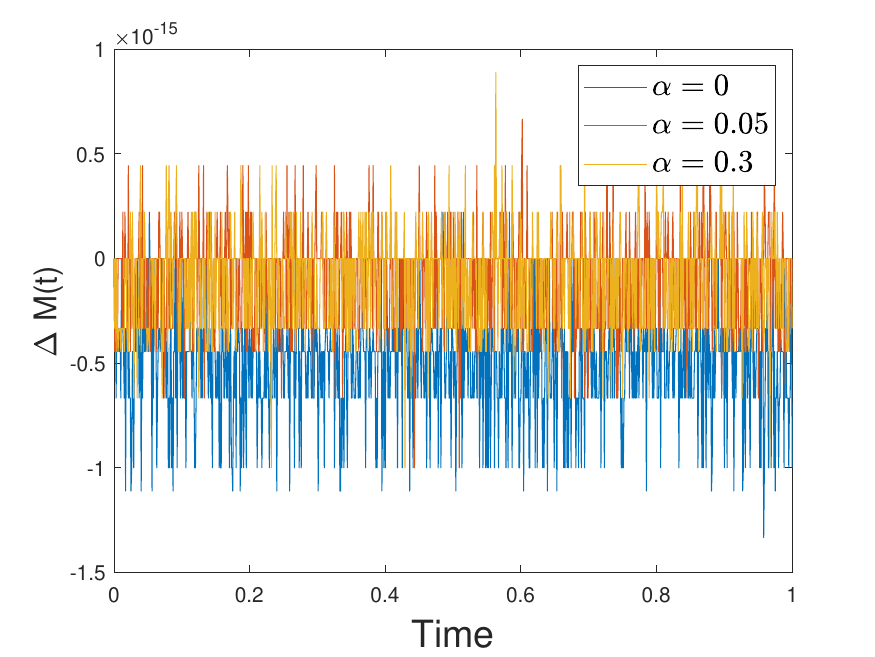}
\includegraphics[width=0.32\textwidth]{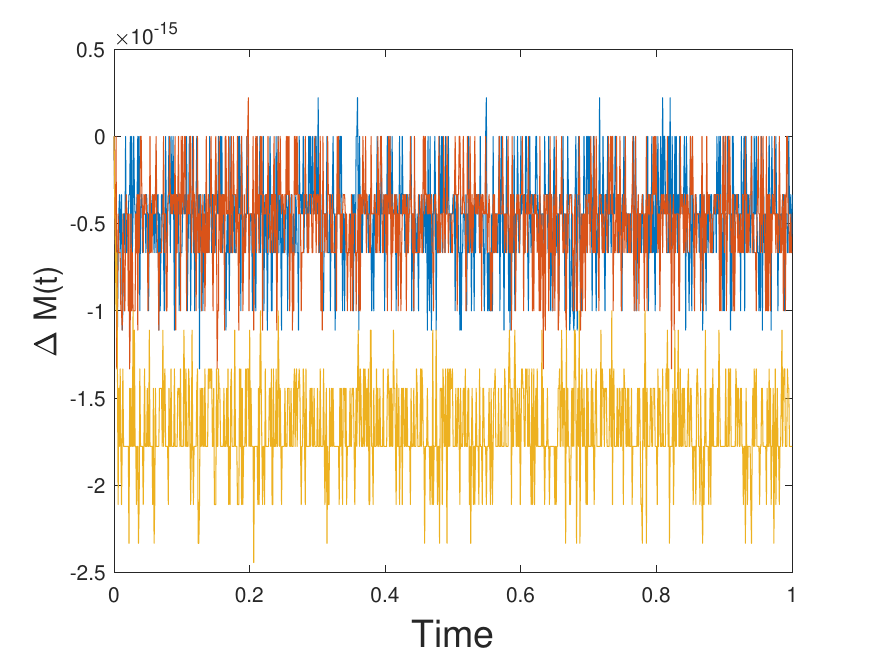}
\includegraphics[width=0.32\textwidth]{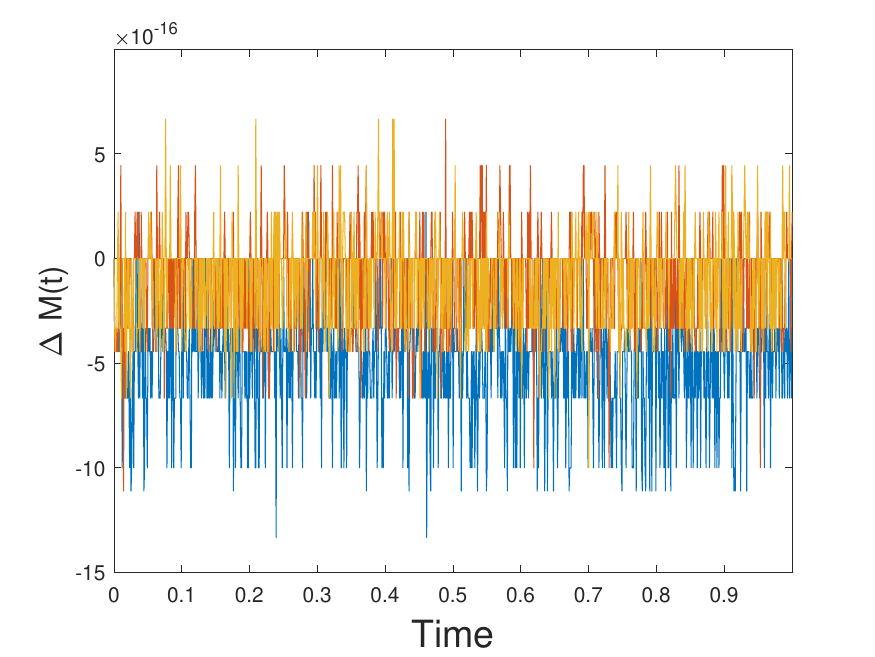}
\textit{(a) The relative error of mass.}
\vspace{\baselineskip} 
\par
\includegraphics[width=0.32\textwidth]{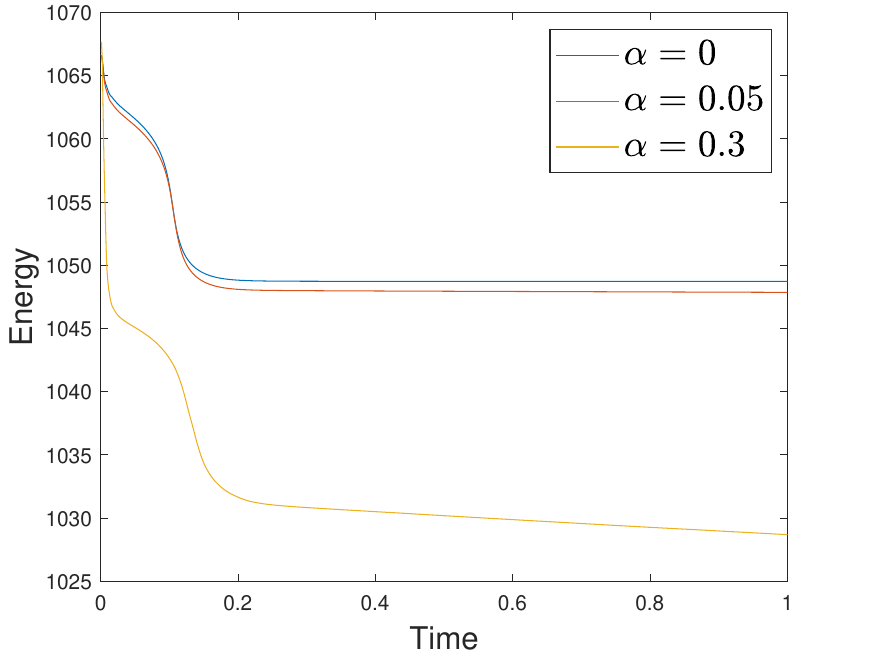}
\includegraphics[width=0.32\textwidth]{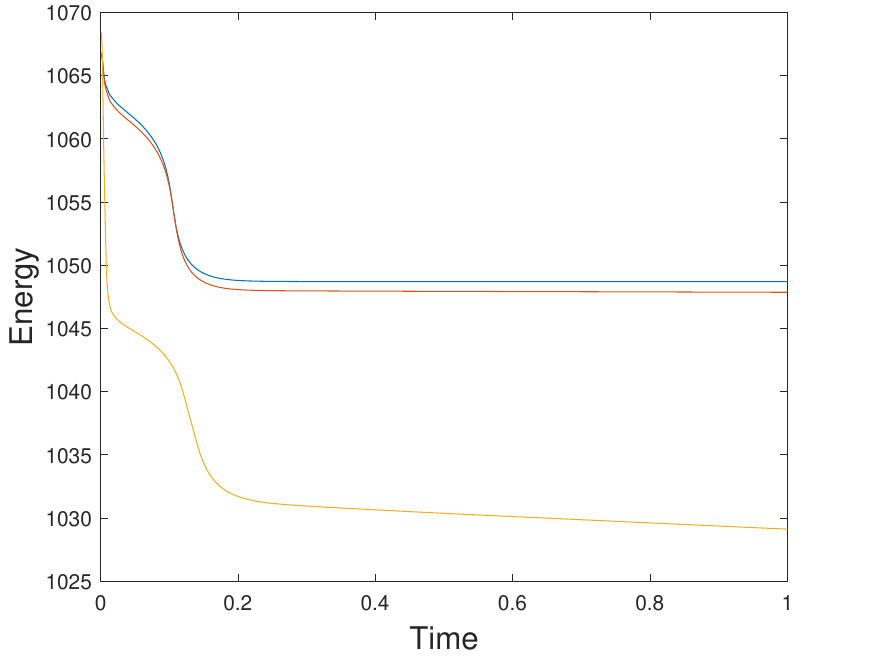}
\includegraphics[width=0.32\textwidth]{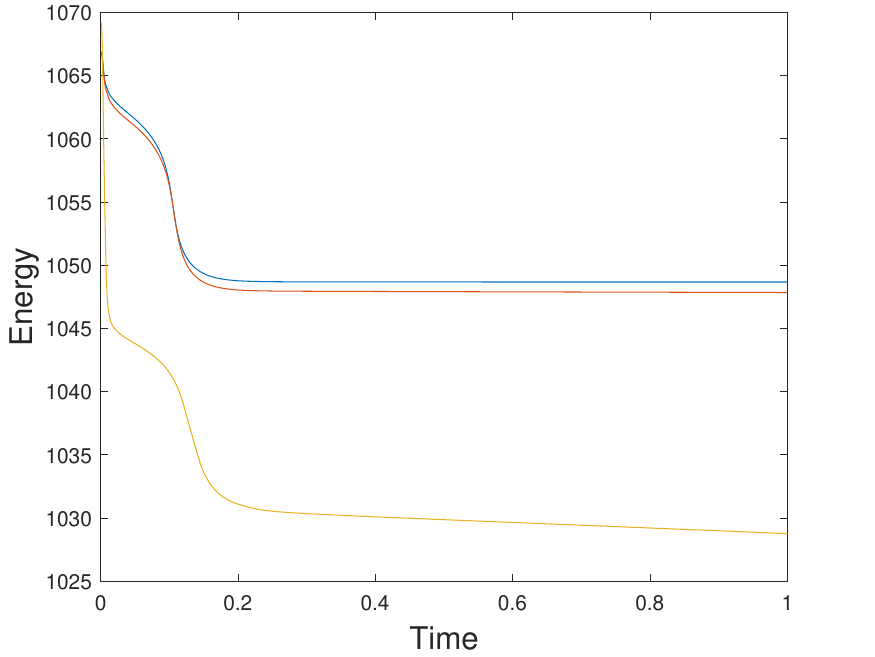}
\textit{(b) The energy evolutions.}
\caption{The relative error of mass and the energy evolutions for $\mathcal{U}_L$-method with different $\theta$:
$\theta=0.5$ (first column), $\theta=0.75$ (second column), $\theta=1$ (last column). The random initial condition \eqref{eqn:initial-condition_2_2} is chosen and the other parameters are selected by \eqref{eqn:parameters}. (a) The relative error of mass with $\tau=1e-3$. (b) The modified energy \eqref{eqn:VE_linear} with $\tau=1e-3$.}
\label{fig:ME_uniform_2D}
\end{figure}

\begin{figure}[!htp]
	\centering
	\includegraphics[width=0.32\textwidth]{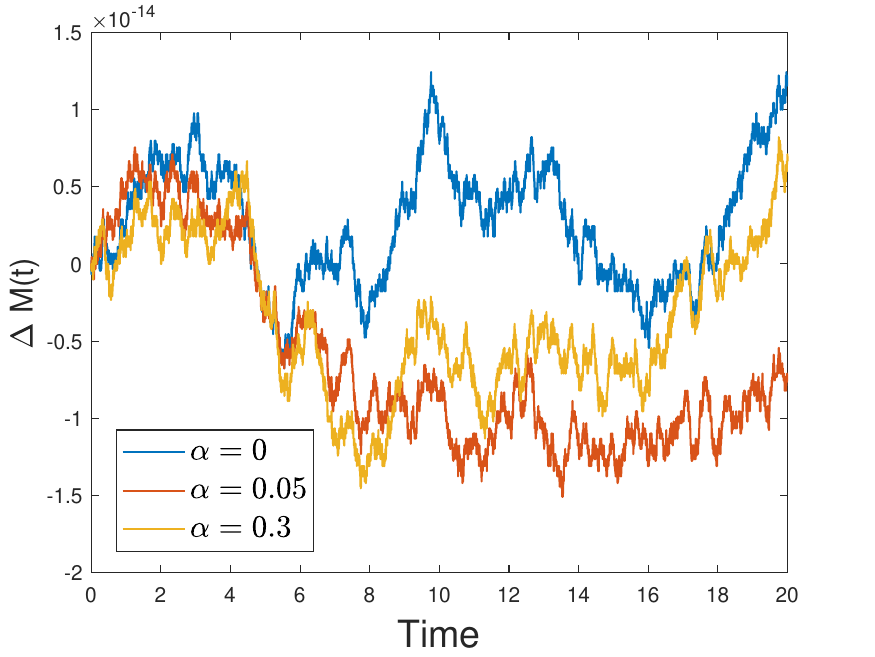}
	\includegraphics[width=0.32\textwidth]{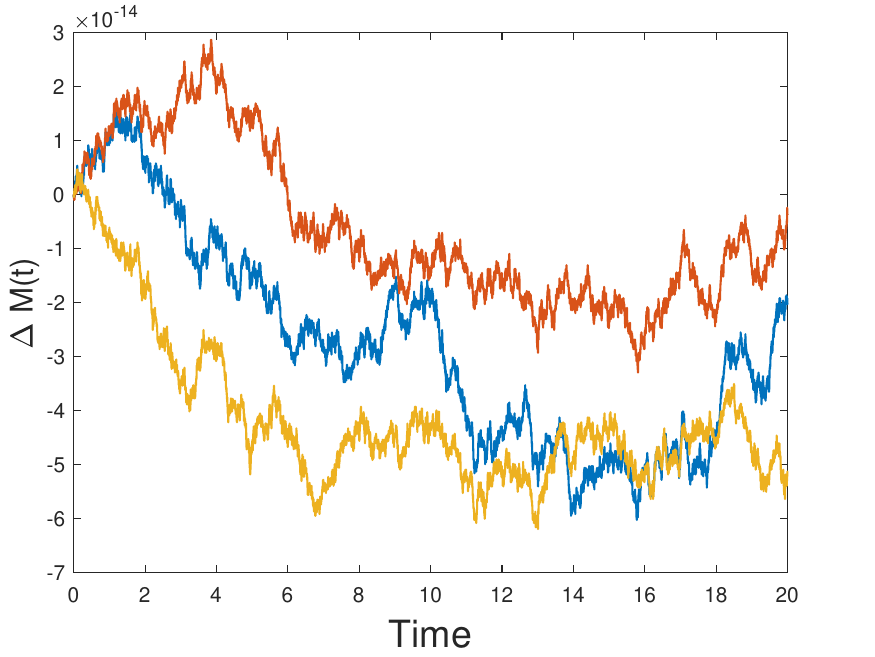}
	\includegraphics[width=0.32\textwidth]{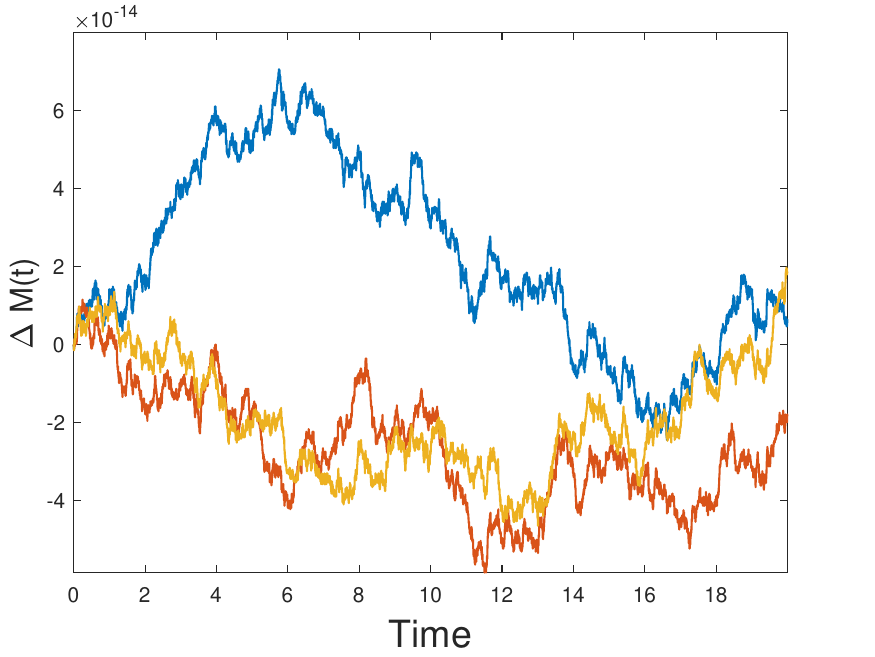}
	\textit{(a) The relative error of mass.}
	\vspace{\baselineskip} 
	\par
	\includegraphics[width=0.32\textwidth]{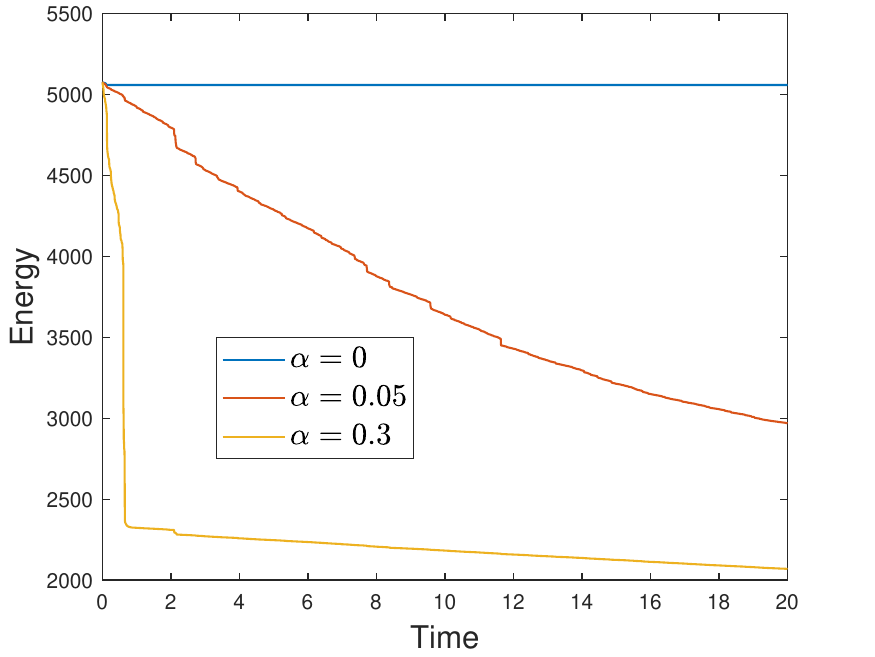}
	\includegraphics[width=0.32\textwidth]{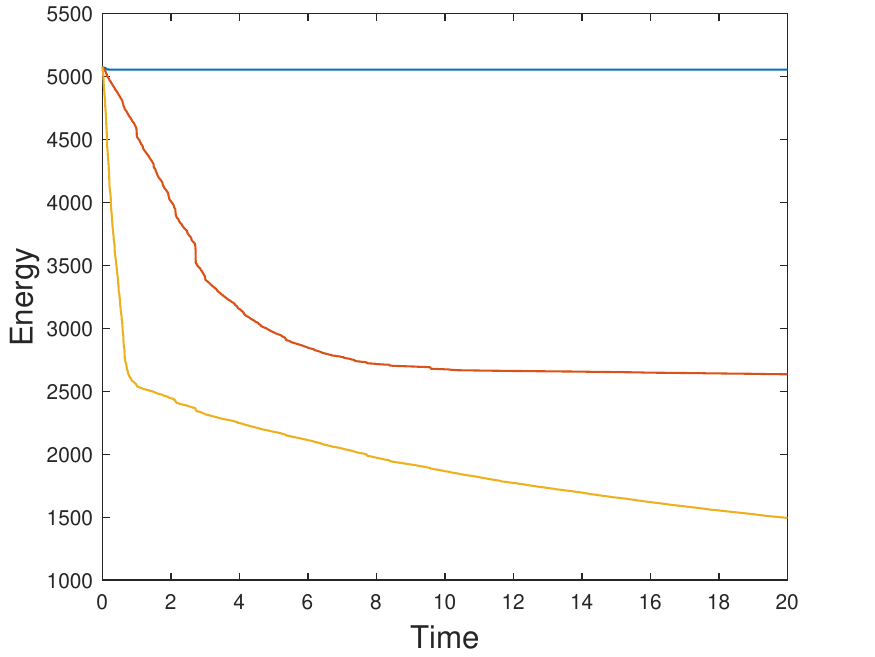}
	\includegraphics[width=0.32\textwidth]{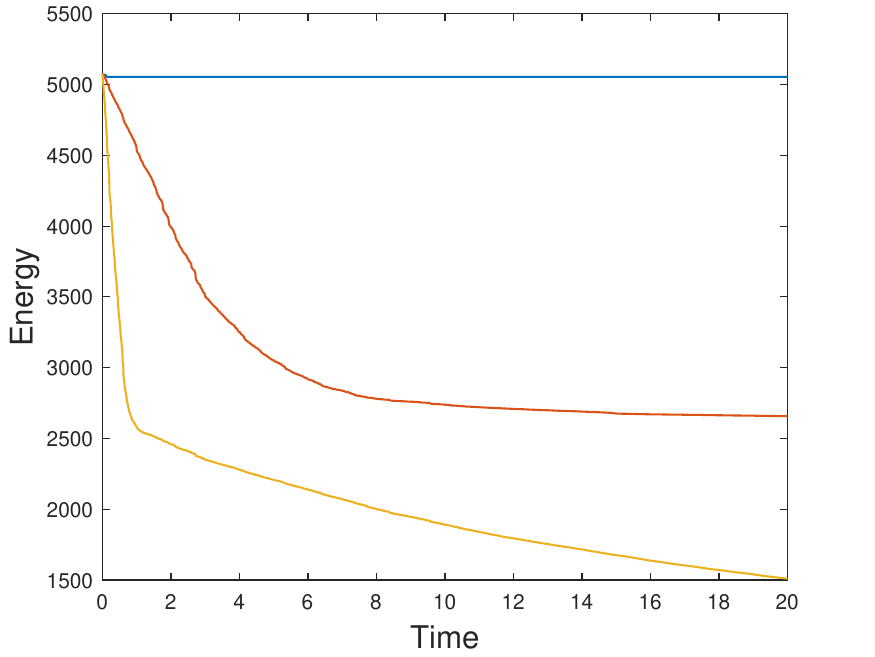}
	\textit{(b) The energy evolutions.}
	\caption{The relative error of mass and the energy evolutions for $\mathcal{V}_L$-method with different $\theta$:
		$\theta=0.5$ (first column), $\theta=0.75$ (second column), $\theta=1$ (last column). The random initial condition \eqref{eqn:initial-condition_2_2} is chosen and the other parameters are selected by \eqref{eqn:parameters}. (a) The relative error of mass with $\tau=1e-3$. (b) The modified energy \eqref{eqn:VE_linear} with $\tau=1e-3$.}
	\label{fig:ME_VTS_2D}
\end{figure}

\begin{figure}[!htp]
	\centering
	\includegraphics[width=\textwidth]{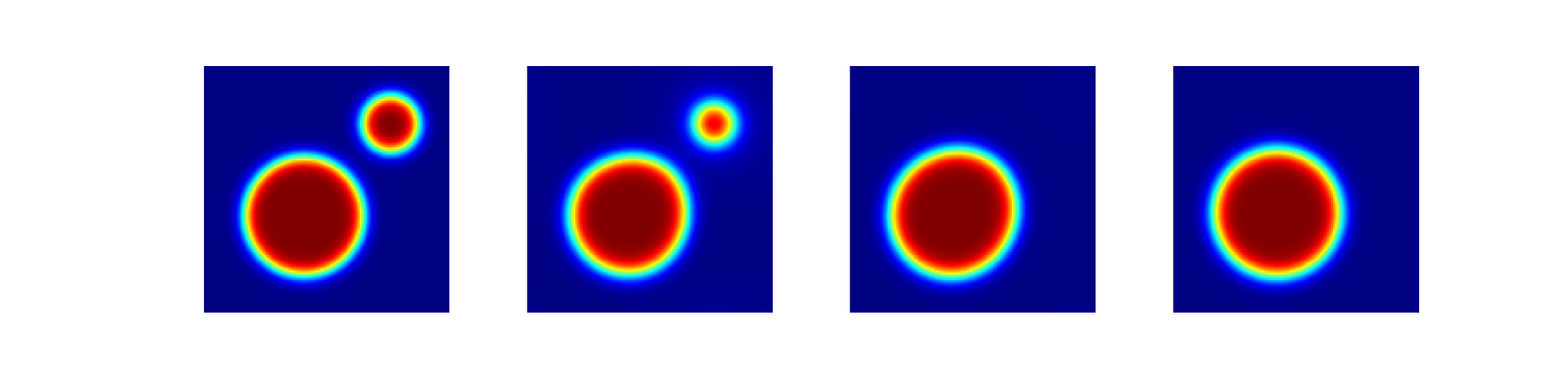}
	\par
	\includegraphics[width=\textwidth]{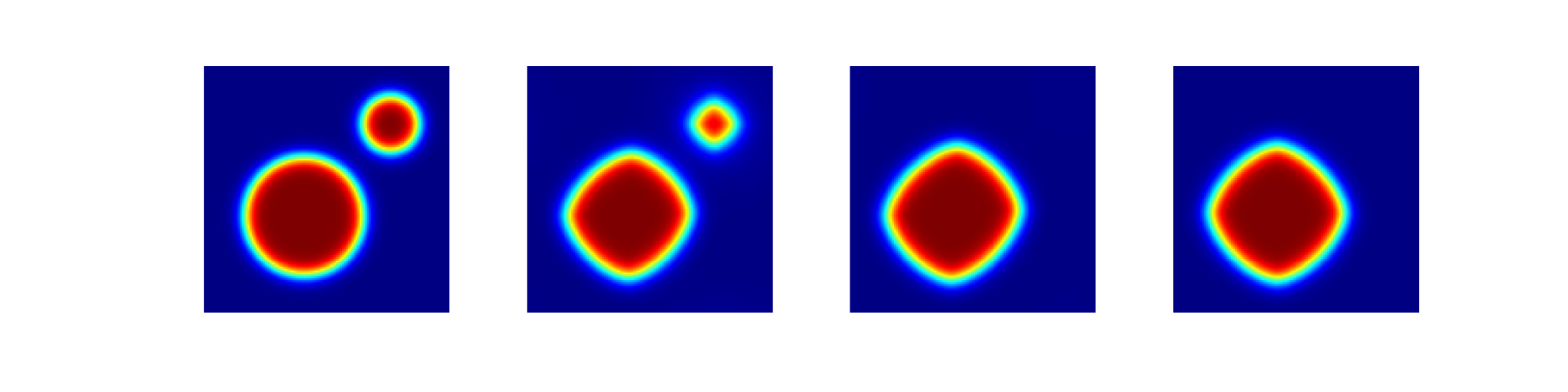}
	\par
	\includegraphics[width=\textwidth]{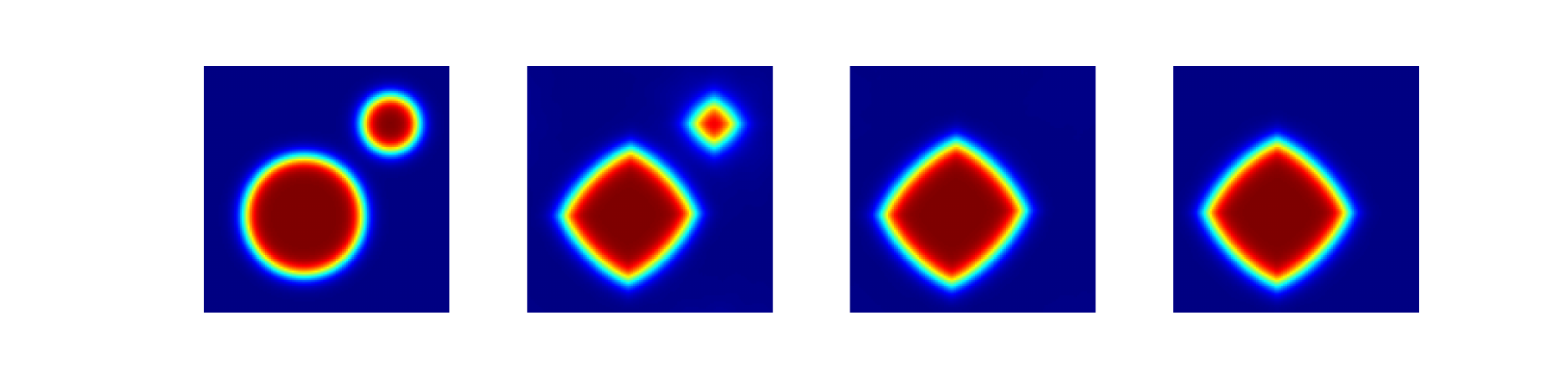}
	\caption{Snapshots of $\phi$ are taken at $0, 6.5e-2, 1.99e-1$ and $2$ for $\mathcal{U}_L$-method with $\tau=1e-3$. The weighted parameter $\theta=0.75$ and the other parameters are selected by \eqref{eqn:parameters}. (a) The 2D dynamical evolution of the phase variable $\phi$ with $\alpha=0$. (b) The 2D dynamical evolution of the phase variable $\phi$ with $\alpha=0.05$. (c) The 2D dynamical evolution of the phase variable $\phi$ with $\alpha=0.1$.}
	\label{figure:snapshot_uniform}
\end{figure}

\begin{figure}[!htp]
	\centering
	\includegraphics[width=\textwidth]{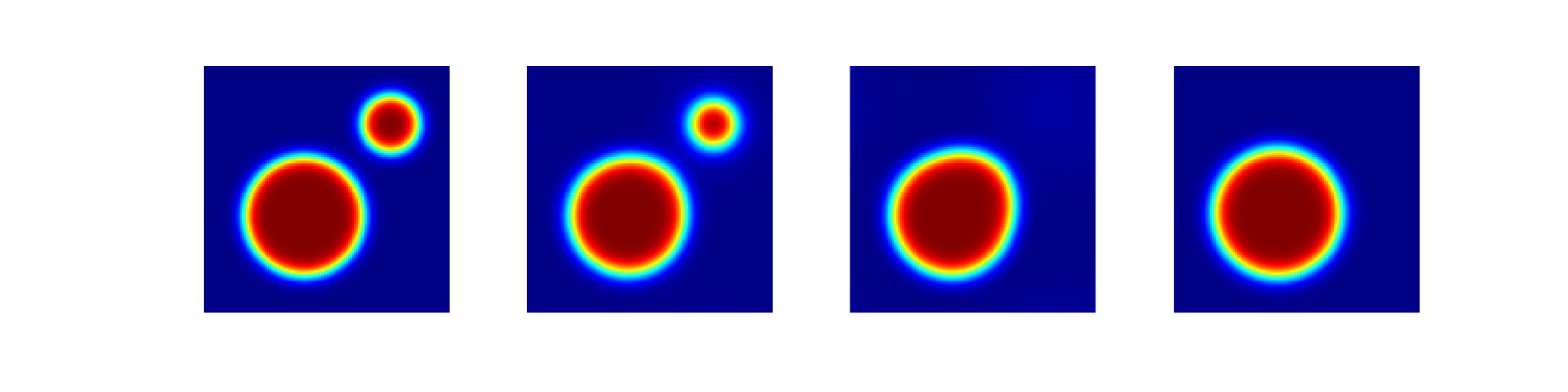}
	\par
	\includegraphics[width=\textwidth]{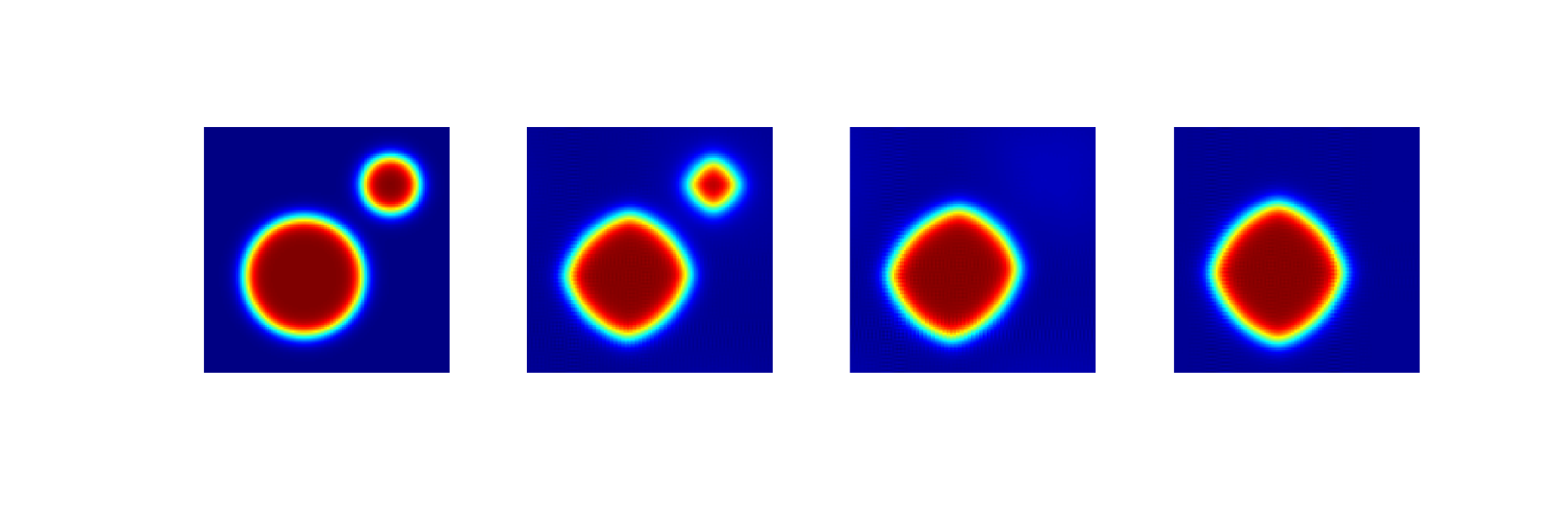}
	\par
	\includegraphics[width=\textwidth]{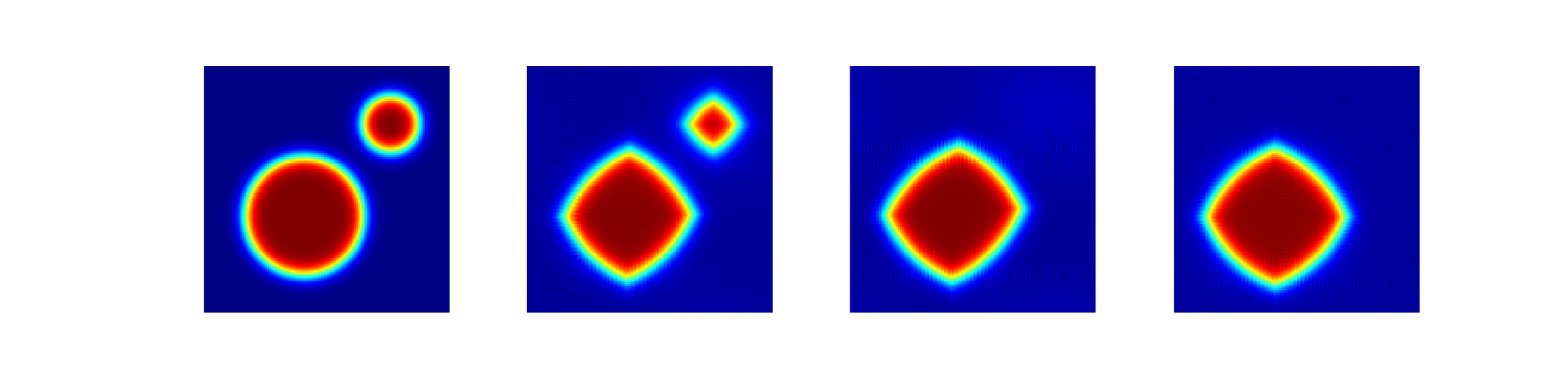}
	\caption{Snapshots of $\phi$, for $\mathcal{V}_L$-method, are taken at $0, 6.5e-2, 1.99e-1$ and $2$ with $\tau=1e-3$, where initial condition \eqref{eqn:initial_condition_3_2}. The weighted parameter $\theta=0.75$ and the other parameters are selected by \eqref{eqn:parameters}. (a) The 2D dynamical evolution of the phase variable $\phi$ with $\alpha=0$. (b) The 2D dynamical evolution of the phase variable $\phi$ with $\alpha=0.05$. (c) The 2D dynamical evolution of the phase variable $\phi$ with $\alpha=0.1$.}
	\label{figure:snapshot_VTS}
\end{figure}

In figures \ref{fig:ME_uniform_2D} and \ref{fig:ME_VTS_2D}, we plot the relative errors of mass and energy with the initial condition \eqref{eqn:initial_condition_3_1} for  $\mathcal{U}_L$-method and $\mathcal{V}_L$-method. which indicate that both these methods maintain structure-preserving properties. Additionally, snapshots of the profiles of the phase-field variable $\phi$ with various anisotropy intensity at different times for $\mathcal{U}_L$-method and $\mathcal{V}_L$-method in Figures \ref{figure:snapshot_uniform} and \ref{figure:snapshot_VTS}. We observe that a coarsening effect in which the smaller circle is absorbed by the larger circle for isotropic system. Furthermore, for isotropic system, the two circles initially evolve into anisotropic shapes, missing orientation at the four corners, followed by the coarsening of the anisotropic system, resulting in the disappearance of the smaller shape. In addition, as the intensity of anisotropy increases, the equilibrium shapes tend to become pyramids with sharper angles from the Figures \ref{figure:snapshot_uniform} and \ref{figure:snapshot_VTS}.

 \textbf{Example 4.}
 In this example, we adopt a random initial condition to simulate the free energy evolution of both an isotropic model and an anisotropic model, and present the temporal evolution of the solutions. Here, we use $\mathcal{U}_L$-method for numerical simulation, and the following random initial condition is chosen below:
 \begin{align}\label{eqn:initial_condition_3_2}
 	\phi(x,y,0)=-0.5+0.001rand(x,y).
 \end{align}

 \begin{figure}[!htp]
 	\centering
 	\includegraphics[width=\textwidth]{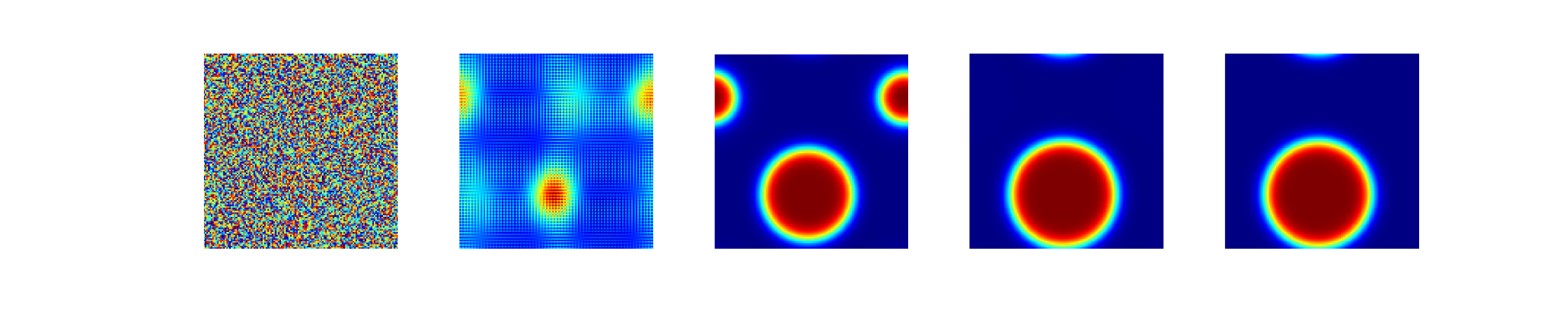}
 	\textit{(a) Snapshots of $\phi$.}\par
 	\vspace{\baselineskip} 
 	\includegraphics[width=0.5\textwidth]{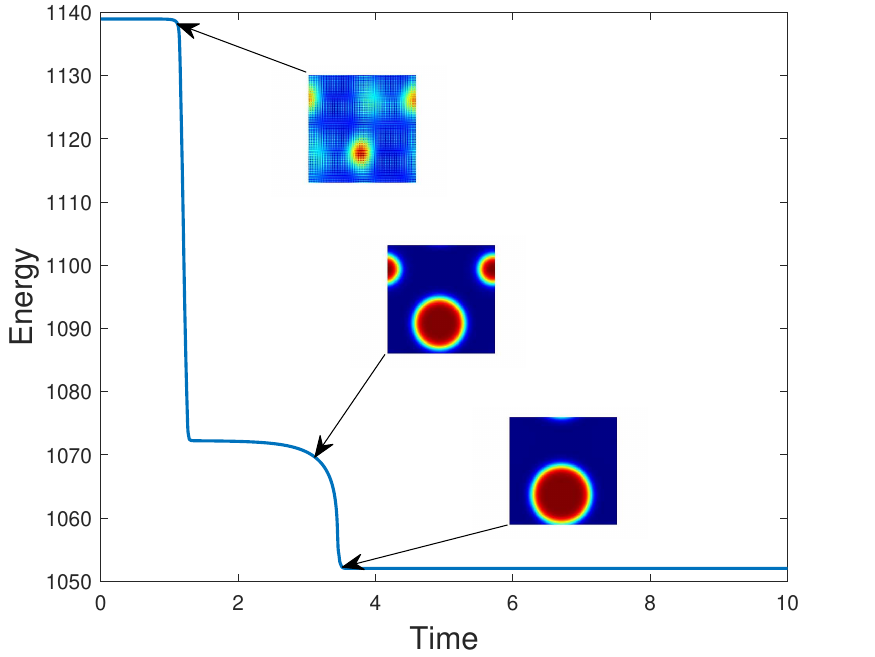}\par
 	\textit{(b) Evolutions of free energy}
 	\caption{(a) The 2D dynamical evolution of the phase variable $\phi$ for the isotropic model ($\alpha=0$) with the linear regularization with the random initial condition \ref{eqn:initial_condition_3_2}, and snapshots are taken at $0, 1.115, 3.12, 3.52$ and $10$.  (b) Time evolution of the free energy \eqref{eqn:UE_linear}.}
 	\label{fig:U_SE_02D}
 \end{figure}

\begin{figure}[!htp]
	\centering
	\includegraphics[width=\textwidth]{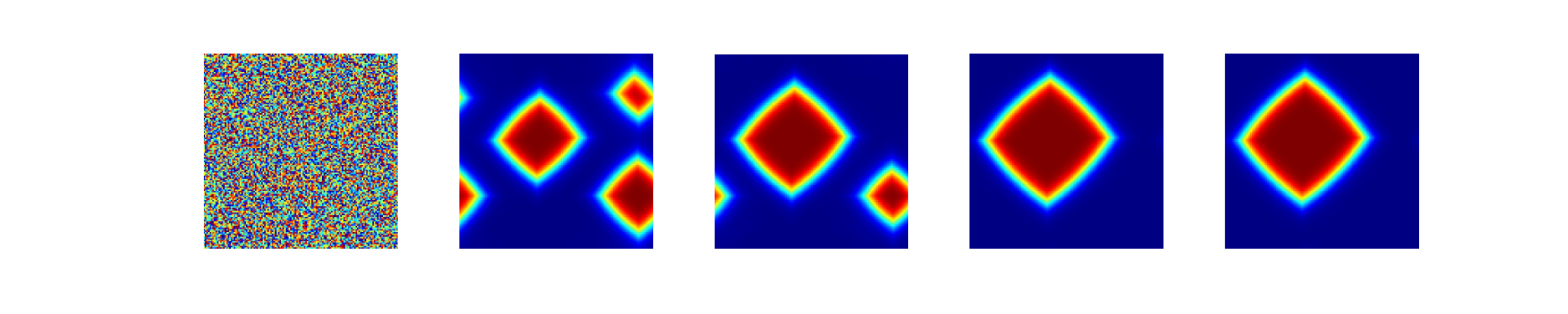}
	\textit{(a) Snapshots of $\phi$.}\par
	\vspace{\baselineskip} 
	\includegraphics[width=0.5\textwidth]{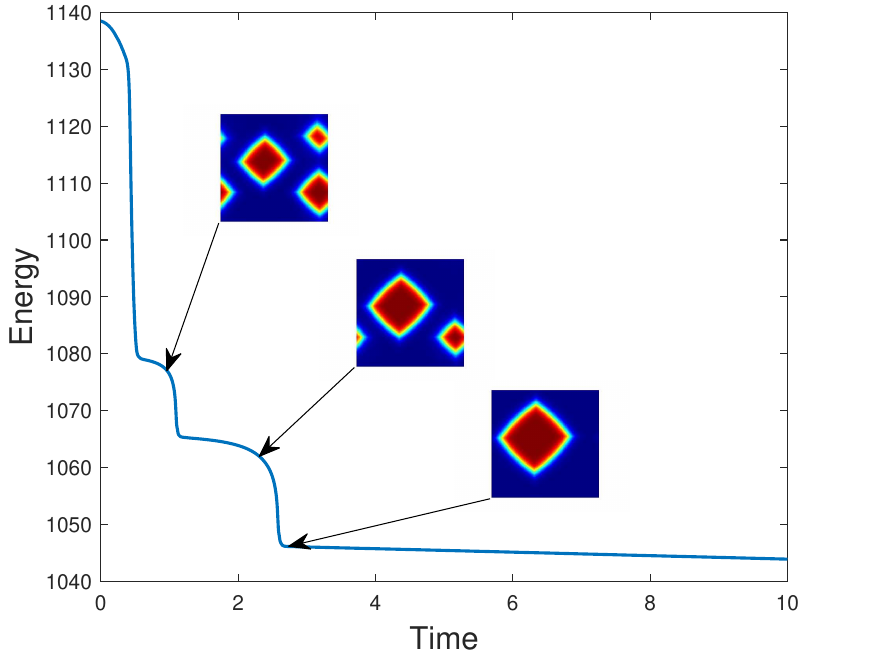}\par
	\textit{(b) Evolutions of free energy}
	\caption{(a) The 2D dynamical evolution of the phase variable $\phi$ for the anisotropic model ($\alpha=0.2$) with the linear regularization with the random initial condition \ref{eqn:initial_condition_3_2}, and snapshots are taken at $0, 0.965, 2.31, 2.74$ and $10$.  (b) Time evolution of the free energy \eqref{eqn:UE_linear}.}
	\label{fig:U_SE_022D}
\end{figure}

 We set the time step $\tau=5e-2$ with the weighted parameter $\theta=0.75$, and select the other parameters by \eqref{eqn:parameters}. For isotropic system, from Figure \ref{fig:U_SE_02D}, we can observe that snapshot of solution with random initial value evolve into multiple circles following the first rapid decline in free energy function. After the second rapid decline, the free energy function reaches a steady state, and the snapshots transform into a single circle. Additionally, in Figure \ref{fig:U_SE_022D}, We see the combined effects of anisotropy and coarsening as time evolves. The free energy function undergoes multiple rapid declines to reach a steady state. Ultimately, the snapshot of $\phi$ evolve into a pyramid at the steady state.


\section{Conclusions}\label{section_6}
In this work, we have proposed WSBDF2 method combined stabilization technique on uniform/nonuniform temporal mesh to solve anisotropic CH models with linear and Willmore regularization, respectively. On uniform temporal mesh, we have employed uniform-time-step WSBDF2 method combined with the traditional SAV approach, and different numerical schemes can be obtained by adjusting the weighted parameter $\theta$. Compared to th existing uniform numerical methods for anisotropic CH models, the proposed WSBDF2 scheme on nonuniform temporal mesh not only well capture the dynamics of the solutions as well as provide a new approach to easily adopt adaptive time stepping methods in future, but also deduce different schemes by change the value of parameter $\theta$. Moreover, the structure-preserving properties of our proposed schemes are rigorously proved.  Finally, extensive numerical experiments validate the correctness and effectiveness of the proposed schemes in theory.

\section{Acknowledgements}
The work is supported by the China
Postdoctoral Science Foundation (No.2023T160589),
National Natural Science Foundation of China (Nos. 11801527,12001499,11971416), Natural Science Foundation of Henan Province (No. 222300420256), Training Plan of Young Backbone Teachers in Colleges of Henan Province (No. 2020GGJS230),  Henan University Science and Technology Innovation Talent support program (No. 19HASTIT025).

\bibliographystyle{elsarticle-num}

\bibliography{thebib}

\end{document}